\documentclass[10pt,reqno]{amsart}
\usepackage{amssymb,amsmath,geometry,color}
\usepackage{graphicx}
\usepackage{epstopdf}
\usepackage{epsfig,multirow}
\usepackage{color,cases}
\usepackage{lipsum}
\usepackage{amsfonts}
\usepackage{graphicx}
\usepackage{epstopdf}
\usepackage{amsopn}
\ifpdf
  \DeclareGraphicsExtensions{.eps,.pdf,.png,.jpg}
\else
  \DeclareGraphicsExtensions{.eps}
\fi

\newtheorem{mytheo}{Theorem}[section]
\newtheorem{mydef}[mytheo]{Definition}

\newtheorem{remark}[mytheo]{Remark}
\newtheorem{mylemma}[mytheo]{Lemma}
\newtheorem{assum}[mytheo]{Assumption}
\numberwithin{equation}{section}

\newcommand{\norm}[1]{\left\Vert#1\right\Vert}
\newcommand{\abs}[1]{\left\vert#1\right\vert}

\newcommand{\sinc}{\textmd{\rm{sinc}}}

\begin{document}

\title[Geometric low-regularity integrator]{Solving  nonlinear Klein-Gordon equation with non-smooth solution by a geometric low-regularity integrator}

\author[B. Wang]{Bin Wang}\address{\hspace*{-12pt}B.~Wang: School of Mathematics and Statistics, Xi'an Jiaotong University, 710049 Xi'an, China}
\email{wangbinmaths@xjtu.edu.cn}

\author[Z. Miao]{Zhen Miao}
\address{\hspace*{-12pt}Z. Miao: School of Mathematics and Statistics,
          Northwestern Polytechnical University,
         710072 Xi'an,  China.}
\email{mz91127@126.com}

\author[Y. L. Jiang]{Yaolin Jiang}
\address{\hspace*{-12pt}Y. L.~Jiang: School of Mathematics and Statistics, Xi'an Jiaotong University, 710049 Xi'an, China}
\email{yljiang@mail.xjtu.edu.cn}
\date{}

\dedicatory{}

\begin{abstract}
 In this paper, we formulate and analyse a geometric low-regularity integrator for solving the nonlinear Klein-Gordon equation in the $d$-dimensional space with $d=1,2,3$. The integrator is constructed based on the two-step trigonometric method and thus it has a simple form. Error estimates are rigorously presented to show that the integrator can achieve second-order time accuracy in the energy space under the regularity requirement in $H^{1+\frac{d}{4}}\times H^{\frac{d}{4}}$. Moreover, the time symmetry of the scheme ensures its good long-time energy, momentum and action  conservations which are rigorously proved by the technique of modulated Fourier expansions. A numerical test is presented and the numerical results demonstrate the superiorities of the new integrator over some existing methods.
\\
{\bf Keywords:}  Geometric low-regularity integrator,  Error estimate, Long-time analysis,  Klein--Gordon equation, Modulated Fourier expansion\\
{\bf AMS Subject Classification:} 35L70, 65M12, 65M15, 65M70.
\end{abstract}

\maketitle


\section{Introduction}\label{sec:introduction}
As an significant  kind of  wave type equations, the nonlinear Klein-Gordon equation plays an important role
in the description of physical phenomena arising from different branches of physics such as plasma physics, particle physics, quantum field theory and cosmology \cite{Bambusi03,Cohen08}.
There has been a growing interest to investigate the numerical solutions of the Klein-Gordon equation and this paper is devoted to    time discrete methods. Various numerical methods in time direction have been well developed and researched  in recent decades such as  trigonometric/exponential integrators \cite{Buchholz2021,Gauckler15,Lubich99}, uniformly accurate methods \cite{Bao14,PI1,S2,Chartier15},   splitting methods \cite{Bao22,Buchholz2018,Faou15} and structure-preserving methods \cite{Chen20,Cohen08-1,LI20,zhao18,W23}.

In this paper, we pay attention to the low-regularity (LR) integrators for the nonlinear Klein--Gordon equation (NKGE):
\begin{equation}\label{klein-gordon}
\begin{aligned}
&\partial_{tt}u(t,x)-\Delta u(t,x)=f\big(u(t,x)\big),\ u(0,x)=\varphi_1(x),\ \partial_{t}u(0,x)=\varphi_2(x),\  0< t\leq T,\ x\in \mathbb{T}\subset \mathbb{R}^d,
\end{aligned}
\end{equation}
which is supplemented with the periodic boundary condition on the
torus $\mathbb{T}$,
where $u(t,x)$ represents the wave displacement at
time $t$ and position $x$, $d=1,2,3$ is the dimension of $x$ and   $f(u)$ is a nonlinear function chosen as the
negative derivative of a potential energy $U(u)\geq 0$. 
The NKGE \eqref{klein-gordon} is  a
Hamiltonian PDE  
with the Hamiltonian $\mathcal{H}(u,v)=\frac{1}{2}\int_{\mathbb{T}}\big(v^2+|\nabla u|^2+2U(u)\big){\rm
d}x,$
which is exactly conserved along the solution $u,v:=\partial_{t}u$ of \eqref{klein-gordon}.  The momentum
$
K(u,v)= \int_{\mathbb{T}}\partial_{x}uvdx$
is also  exactly conserved by the solution.

For the system  \eqref{klein-gordon}, it is well known that  if the initial data $(\varphi_1,\varphi_2)$ is in the energy space $H^1(\mathbb{T})\times L^2(\mathbb{T})$, the solution $(u,\partial_tu)$
is also   bounded in the energy space  for any $0< t\leq T$.
However, due to the lack of sufficient regularity which ensures the boundedness of high-order derivatives of the solution, classical time-stepping second-order methods such as some splitting methods and trigonometric integrators  approximate the solution  $(u,\partial_tu)$  with second-order convergence only in
the weaker space $L^2(\mathbb{T})\times H^{-1}(\mathbb{T})$. If the  second order convergence  is given in the energy space $H^1(\mathbb{T})\times L^2(\mathbb{T})$,  the
initial data is generally required  to be in the stronger space $H^2(\mathbb{T})\times H^1(\mathbb{T})$.
To overcome this barrier,  low-regularity  (LR)  integrators have been studied in recent years  \cite{S1,LI22,LI23,SU23,O18,O21,O22,R21,Zhao23}. For the system of NKGE, the method presented in \cite{R21} can have second-order convergence in the energy space $H^1(\mathbb{T})\times L^2(\mathbb{T})$ under the weaker regularity condition
 $(\varphi_1 ,\varphi_2 )\in H^{\frac{7}{4}}(\mathbb{T})\times H^{\frac{3}{4}}(\mathbb{T})$.
For a special nonlinear case $f(u)=u^3$ in one dimension $d=1$, a symmetric low-regularity  integrator was derived in \cite{Zhao23} and the method has second-order convergence in $H^{r}(\mathbb{T}) \times H^{r-1}(\mathbb{T})$ under the same regularity condition $(\varphi_1 ,\varphi_2 )\in H^{r}(\mathbb{T}) \times H^{r-1}(\mathbb{T})$. More recently, for the system  \eqref{klein-gordon} with a general nonlinear function $f(u)$, a second-order
low-regularity  integrator was given in \cite{LI23} and it can get second order convergence in  $H^1(\mathbb{T})\times L^2(\mathbb{T})$ under the weaker regularity condition
 $(\varphi_1 ,\varphi_2 )\in H^{1+\frac{d}{4}}(\mathbb{T})\times H^{\frac{d}{4}}(\mathbb{T})$ with $d=1,2,3.$

Up to now it seems that all the low-regularity  integrators  are more complicated than the traditional algorithms
since they are usually designed by a complicated process.  
 Moreover, a pity within most of LR  integrators so far is the lack of long-time analysis on the behaviour in the conservation laws.  The main difficulty lies in the complicated scheme of LR  integrators which brings great  constriction in the analysis.
The aim of this paper lies in approaching solutions to these problems:  we will formulate a simple geometric  low-regularity  integrator  and analyse its long time behaviour for  the nonlinear Klein--Gordon equation \eqref{klein-gordon} with a general nonlinear function $f(u)$ and the dimension $d=1,2,3$. The central idea to achieve this is embedding the symmetry in time and two-step trigonometric method  into the numerical discretisation.  In the analysis, the key step lies in
the full use of  the structure in the Klein–Gordon equation to reduce regularity and the essential
 technique of  modulated Fourier expansions \cite{Cohen08,Cohen08-1,hairer2000,Cohen08-0,hairer2006,WZ21}.
The new integrator will in particular have the following novelties:

 (I) The integrator is a traditional two-step trigonometric method and thus its scheme is very simple. In each step of the calculation, the scheme only needs to calculate the nonlinear function $f(u)$ once without any other nonlinear terms. Thus the proposed integrator has very small calculation cost compared with the existing low-regularity  methods in the  literature.
(II) The proposed integrator  has  low-regularity property, i.e., it can have second order convergence in  $H^1(\mathbb{T})\times L^2(\mathbb{T})$ under the weaker regularity condition
 $(\varphi_1,\varphi_2)\in H^{1+\frac{d}{4}}(\mathbb{T})\times H^{\frac{d}{4}}(\mathbb{T})$ with $d=1,2,3.$
(III) We successfully combine the time symmetry with this  LR integrator and establish a  rigorous long-time  analysis for the  proposed scheme. By  the technique of  modulated Fourier expansions,  long-time almost  conservations are shown for the proposed geometric  low-regularity integrator.


The rest of this paper is organized as follows. In Section \ref{sec:symmetric LR}, we present the formulation of  a symmetric low-regularity integrator and carry out a numerical test to show its superiorities in comparison with some existing methods in the literature.  The error estimates of the proposed scheme are given in Section \ref{sec:3} and long-time analysis is  established in Section \ref{sec:4}.

\section{Geometric  low-regularity  integrator and numerical test}\label{sec:symmetric LR}
Before getting into the formulation of integrator, we introduce some notations which will be used in the whole paper.
Define the operator
$\mathcal{A}$ by  $
(\mathcal{A}w)(x):=-\Delta  w(x)$ with $w(x)$ on $\mathbb{T}$. For simplicity, we denote
\begin{equation}\label{notations w}\begin{aligned}&\sinc(t):=\frac{\sin(t)}{t},\ \alpha(t):=\left(
              \begin{array}{c}
                \cos(t\sqrt{\mathcal{A}})  \\
                t \sinc( t\sqrt{\mathcal{A}})  \\
              \end{array}
            \right),\   \beta(t):=\left(
              \begin{array}{c}
                t \sinc( t\sqrt{\mathcal{A}})    \\
                \cos(t\sqrt{\mathcal{A}}) \\
              \end{array}
            \right),\   \gamma(t): =\left(
              \begin{array}{c}
                \sqrt{\mathcal{A}} \sin(t\sqrt{\mathcal{A}})    \\
                \cos(t\sqrt{\mathcal{A}}) \\
              \end{array}
            \right).
\end{aligned}\end{equation}
We also suppress the $x$-dependence of the unknown
functions   $u(t), v(t)$. The  time stepsize used in the integrator is denoted by  $0<h <1$ and for the time we consider $t_n=n h $ for $n\in \mathbb{N}^{+}$. We denote
    the Sobolev space by $H^{\nu}(\mathbb{T})$ with  $\nu\geq 0$ and its  norm   is referred as $\norm{\cdot}_{H^{\nu}}$.
    We shall denote $A\lesssim B$ for $A\leq CB$, where $C>0$ is a generic constant independent of $h $ and $n$.
\subsection{Explicit symmetric low-regularity  integrator}
In the formulation, the Duhamel’s formula for   the nonlinear Klein--Gordon
equation \eqref{klein-gordon} is a key theoretical tool which reads
\begin{equation}\label{Duhamel}
\begin{aligned}
u(t_n+s)=&\cos(s\sqrt{\mathcal{A}}) u(t_n)+s \sinc( s\sqrt{\mathcal{A}})  v(t_n)+\int_0^{s } (s-\theta)\sinc( (s-\theta)\sqrt{\mathcal{A}}) f\big( u(t_n+\theta)\big)d\theta,\\
v(t_n+s )=&-s\mathcal{A} \sinc( s\sqrt{\mathcal{A}}) u(t_n)+\cos(s\sqrt{\mathcal{A}})  v(t_n)+\int_0^{s }  \cos( (s-\theta)\sqrt{\mathcal{A}}) f\big( u(t_n+\theta)\big)d\theta,
\end{aligned}
\end{equation}
at $t=t_n+s$ with $s\in \mathbb{R}$.
Letting $s=\pm h $ in the  formulae   \eqref{Duhamel} and  then combining them by
some calculations, we obtain
\begin{equation}\label{u2add}
\begin{aligned}
u(t_{n+1})&+u(t_{n-1})
=2\cos(h \sqrt{\mathcal{A}}) u(t_n)+\int_0^{h} (h -\theta)\sinc( (h -\theta)\sqrt{\mathcal{A}}) \Big(f\big( u(t_n+\theta)\big)+f\big( u(t_n-\theta)\big)\Big)d\theta,\\
v(t_{n+1})&-v(t_{n-1})=-2h \mathcal{A} \sinc( h \sqrt{\mathcal{A}}) u(t_n)+\int_0^{h  }  \cos( (h -\theta)\sqrt{\mathcal{A}}) \Big(f\big( u(t_n+\theta)\big)+f\big( u(t_n-\theta)\big)\Big)d\theta.
\end{aligned}
\end{equation}
By choosing some numerical approximations of the integrals appeared above, we define the following two-step method for solving the Klein--Gordon
equation \eqref{klein-gordon}.
\begin{mydef}\label{def-se} (\textbf{Symmetric  low-regularity  integrator})
Choose a time stepsize $0<h <1$ and denote the numerical solution as  $u_{n}\approx u(t_{n},x)$ and   $v_{n}\approx v(t_{n},x)$ at  $t_n=nh $. The two-step scheme for solving the   Klein--Gordon
equation \eqref{klein-gordon} is defined as
 \begin{equation}\label{method}
\begin{aligned}
u_{n+1}=&2\cos(h \sqrt{\mathcal{A}}) u_n-u_{n-1}+h ^2\sinc(  h \sqrt{\mathcal{A}}) f( u_n),\\
v_{n+1}=&-2h \mathcal{A} \sinc( h \sqrt{\mathcal{A}}) u_n+v_{n-1}+h \big(\cos(h \sqrt{\mathcal{A}}) +\sinc(  h \sqrt{\mathcal{A}})\big) f( u_n),
\end{aligned}
\end{equation}
where  $ n=1,2,\ldots,T/h -1,$ $u_0=\varphi_1(x)$ and $ v_0=\varphi_2(x)$.
The other starting value $u_1,v_1$ is obtained by the LR method  proposed in \cite{LI23}:
 \begin{equation}\label{start value}
\begin{aligned}
u_{1}=&\cos(h \sqrt{\mathcal{A}}) u_0+h  \sinc( h \sqrt{\mathcal{A}})  v_0+\frac{h ^2}{2} \sinc( h \sqrt{\mathcal{A}}) f( u_0) +h  \frac{ \sinc( h \sqrt{\mathcal{A}})-\cos(h \sqrt{\mathcal{A}})}{2\mathcal{A}}f'( u_0)v_0,\\
v_{1}=&-h \mathcal{A} \sinc( h \sqrt{\mathcal{A}}) u_0+\cos(h \sqrt{\mathcal{A}}) v_{0}+h
\frac{\cos(h \sqrt{\mathcal{A}}) + \sinc( h \sqrt{\mathcal{A}})}{2}  f( u_0)  +
\frac{h ^2\sinc( h \sqrt{\mathcal{A}})}{2}f'( u_0)v_0,
\end{aligned}
\end{equation}
 where the prime on $f$ indicates the   derivative  of $f(u)$ w.r.t. $u$.
\end{mydef}

\begin{remark}
It can be seen clearly that the explicit  scheme \eqref{method} is very simple and the computation cost is very low which will be shown by the numerical results of next subsection. Although
\eqref{start value} is not so concise as \eqref{method}, we just use it once to get the starting value and thus  the computational complexity is basically unaffected. Moreover,  the  integrator \eqref{method}  is obviously a time symmetric method, i.e., by switching $h  \leftrightarrow -h $ and  $n+1 \leftrightarrow n-1$, the scheme \eqref{method} remains the same. Thus  a good long-time behavior of this scheme is inherited which will be rigorously studied in  Section \ref{sec:4}.
\end{remark}
%
%
\subsection{Numerical test}\label{sec:experiments}
Now we show the numerical performance of the proposed integrator by   a numerical test. We consider
the NKGE \eqref{klein-gordon} with $d=1,\mathbb{T}=(-\pi,\pi), f(u)=\sin(u)$.
The initial values $ \varphi_1(x)$ and $\varphi_2(x)$ are chosen   as described in Section 5.1 of \cite{O18} and Section 4 of \cite{Zhao23}.
This random initial data  is
from the space  $H^{\theta}(\mathbb{T})\times H^{\theta-1}(\mathbb{T})$ and has the bound $\norm{\varphi_1(x)}_{H^1}=1$ and $\norm{\varphi_2(x)}_{L^2}=1$. We use the fourier spectral collocation method to the  spatial domain with  grid points $x_j=j\pi/N_{x}$ for $j=-N_{x},-N_{x}+1,\ldots,N_{x}-1$.
   \begin{figure}[t!]
\centering
\includegraphics[width=4.5cm,height=4.3cm]{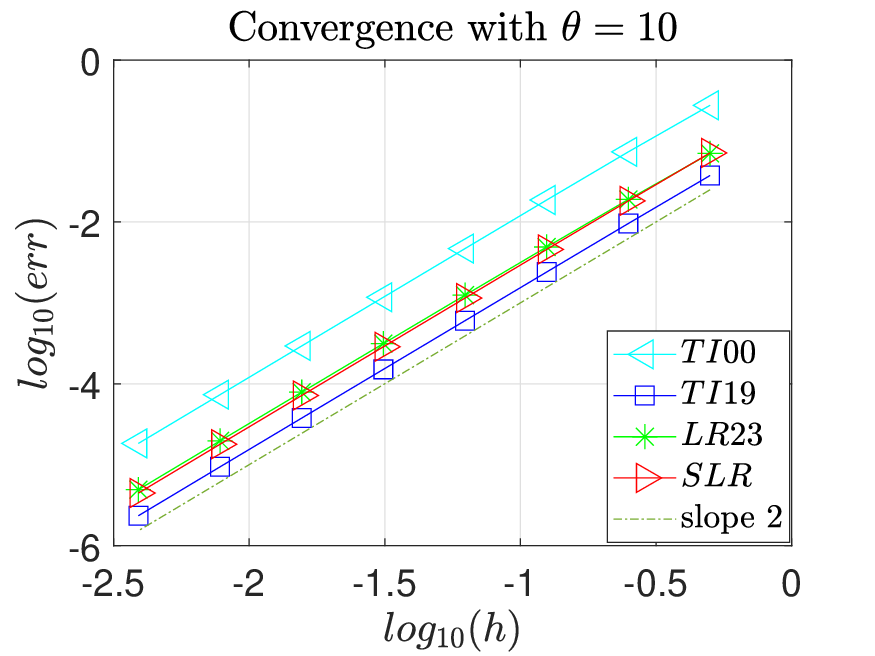}
\includegraphics[width=4.5cm,height=4.3cm]{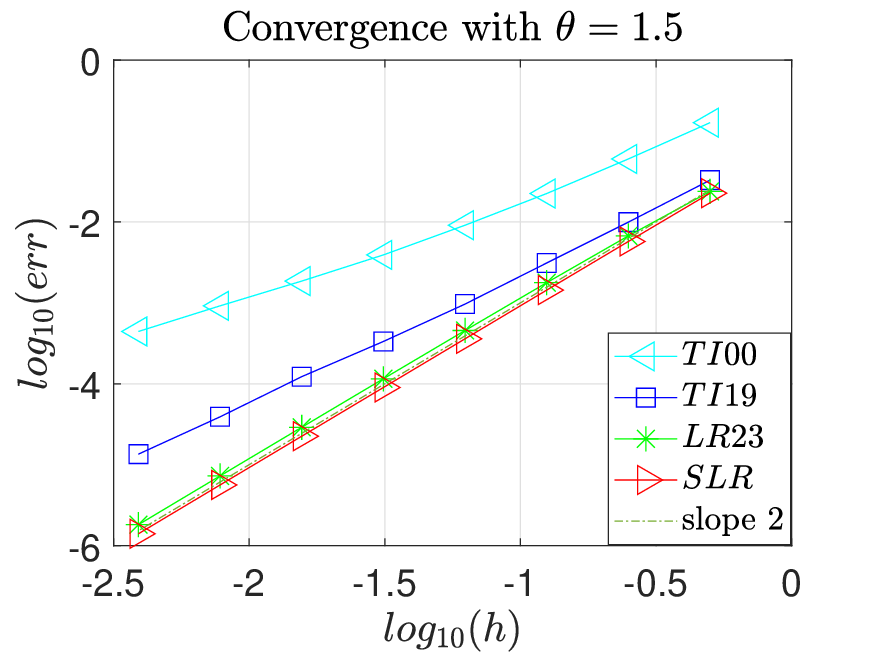}
\includegraphics[width=4.5cm,height=4.3cm]{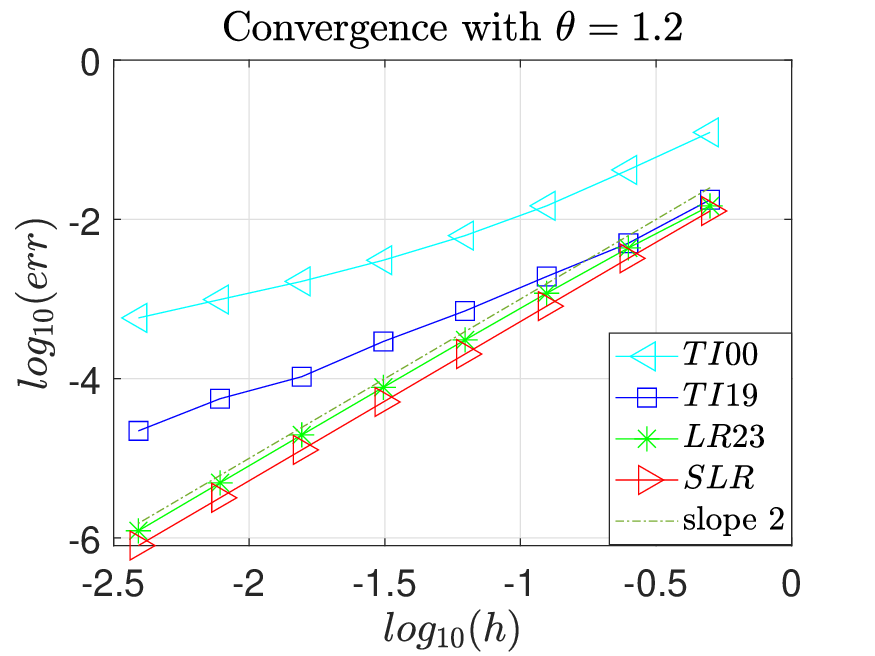}\\
\includegraphics[width=4.5cm,height=4.3cm]{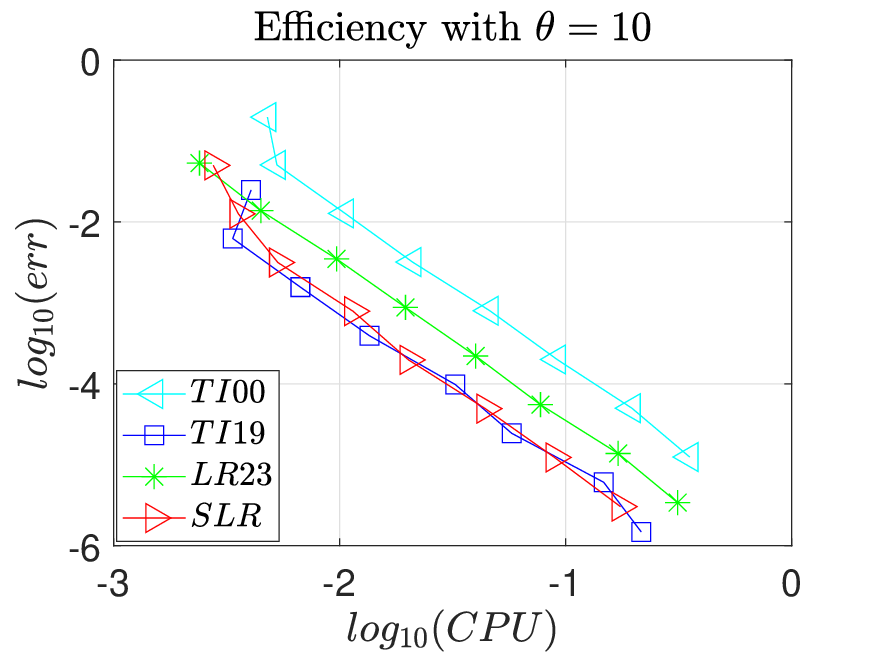}
\includegraphics[width=4.5cm,height=4.3cm]{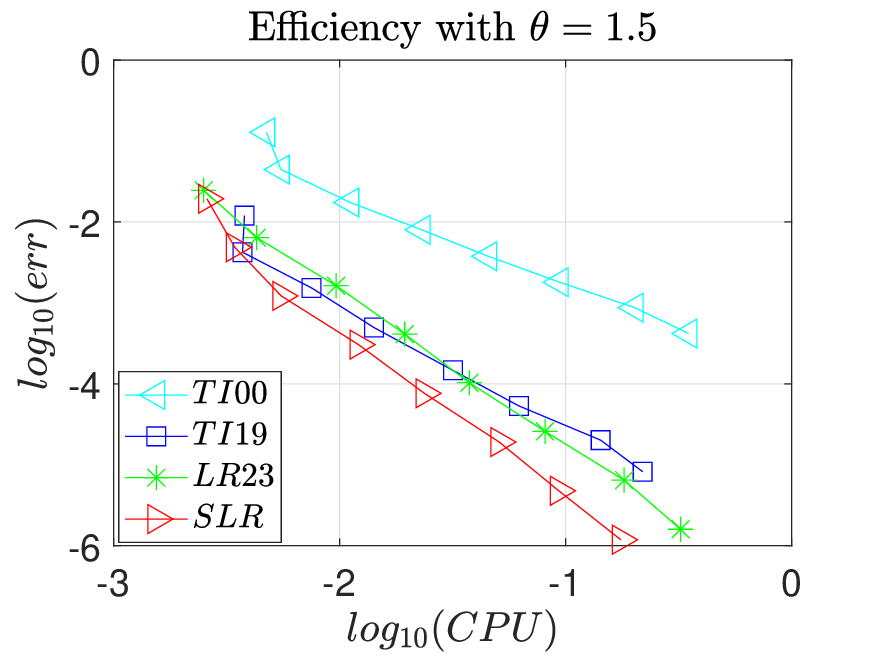}
\includegraphics[width=4.5cm,height=4.3cm]{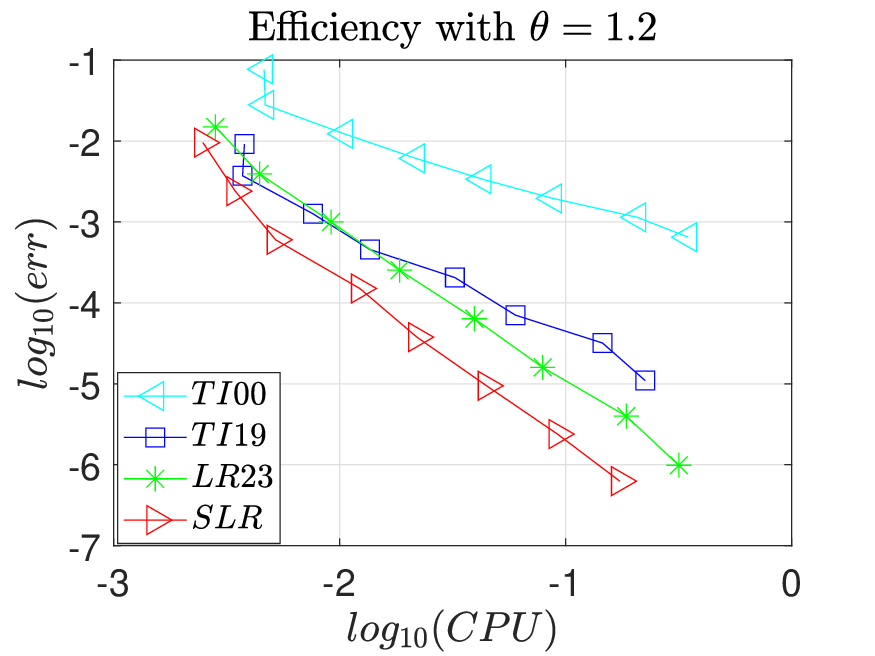}
\caption{Errors of the numerical solutions with  initial data $H^{\theta}(\mathbb{T}) \times H^{\theta-1}(\mathbb{T})$ against $h $  (1st row)  and CPU time (2nd row) with $h =1/2^k$, where $k=2,3,\ldots,9$.} \label{p1}
\end{figure}

\textbf{Accuracy.}  For comparison, we consider
the classical   second-order trigonometric integrators from \cite{hairer2000}  and  \cite{W23} and we respectively denote them by
TI00 and TI19. Meanwhile, we choose the second-order low-regularity  method derived in  \cite{LI23} and denote it by LR23. The novel symmetric low-regularity  method proposed in this paper is referred as SLR.
In this test,
we apply the methods to
solve this system with $N_x=64$ and  display  the global errors
 $err=\frac{\norm{u_n-u(t_n)}_{H^{1}}}{\norm{u(t_n)}_{H^{1}}}+\frac{\norm{v_n-v(t_n)}_{L^{2}}}{\norm{v(t_n)}_{L^{2}}}$ at $t_n=1$. The results for the initial value $(\varphi_1,\varphi_2)\in H^{\theta}(\mathbb{T})\times H^{\theta-1}(\mathbb{T})$ with different $\theta$ are given  in Figure \ref{p1} (1st row). In the light of these results, we have the following observations. For the strong-regularity  initial data ($\theta=10$), all the methods performance second order convergence. However, for the  low-regularity initial data ($\theta=1.5, 1.2$),  only the new integrator SLR proposed in this article and the method LR23 derived in  \cite{LI23} have second-order convergence in $H^{1}(\mathbb{T})\times L^{2}(\mathbb{T})$.

 \textbf{Efficiency.} Compared with the existing methods, the scheme of our integrator given in this paper is very simple.
Only one calculation of  the nonlinear function $f(u)$ is required  in each computation step. To show the efficiency of our scheme, we
 solve this problem on  $[0,5]$ and present the  log-log plot of the temporal error   against CPU time
 in Figure \ref{p1}  (2nd row). It can be seen clearly that  our integrator needs very low computation cost and  has very competitive efficiency.

   \begin{figure}[t!]
\centering
\includegraphics[width=4.5cm,height=4.3cm]{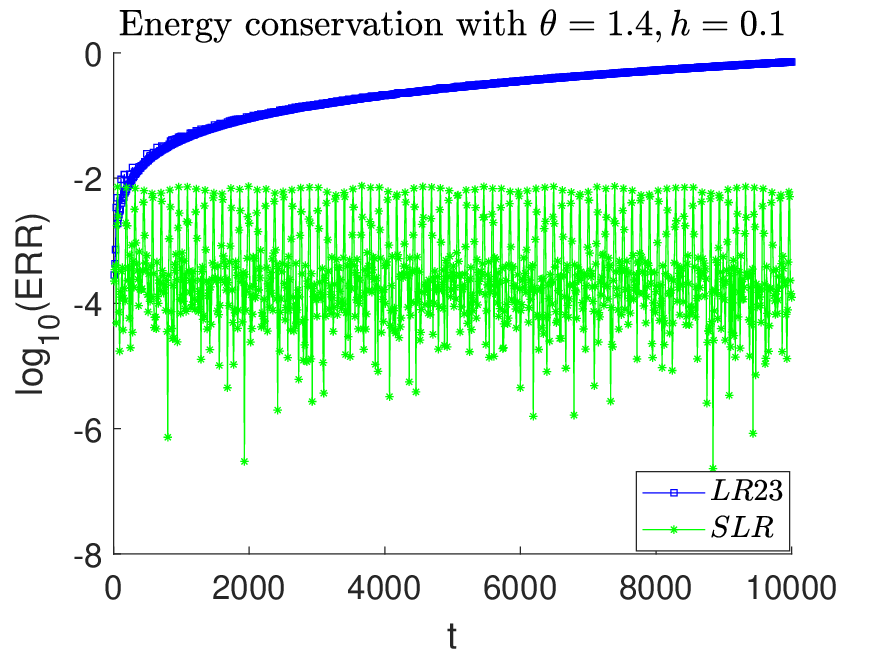}
\includegraphics[width=4.5cm,height=4.3cm]{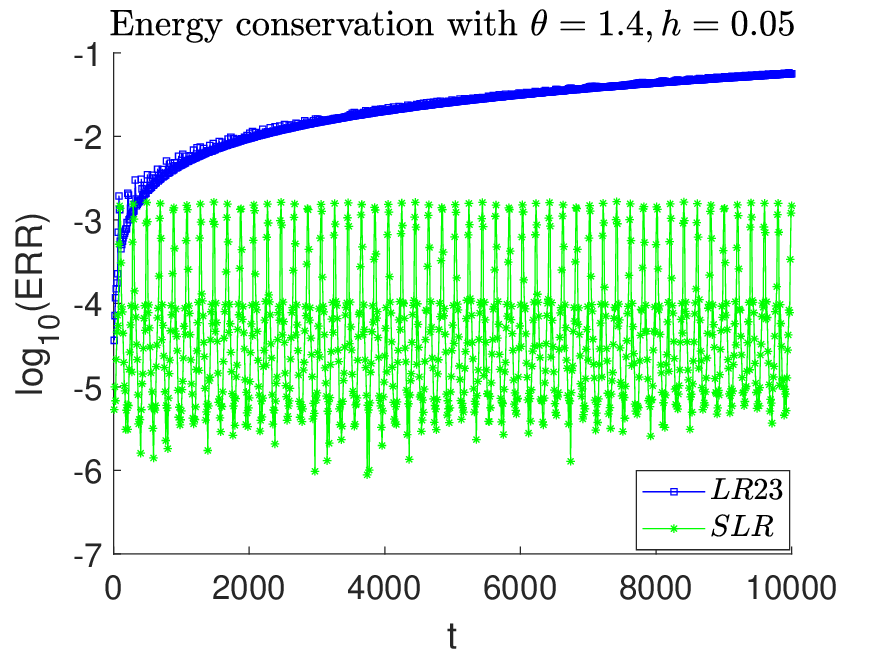}
\includegraphics[width=4.5cm,height=4.3cm]{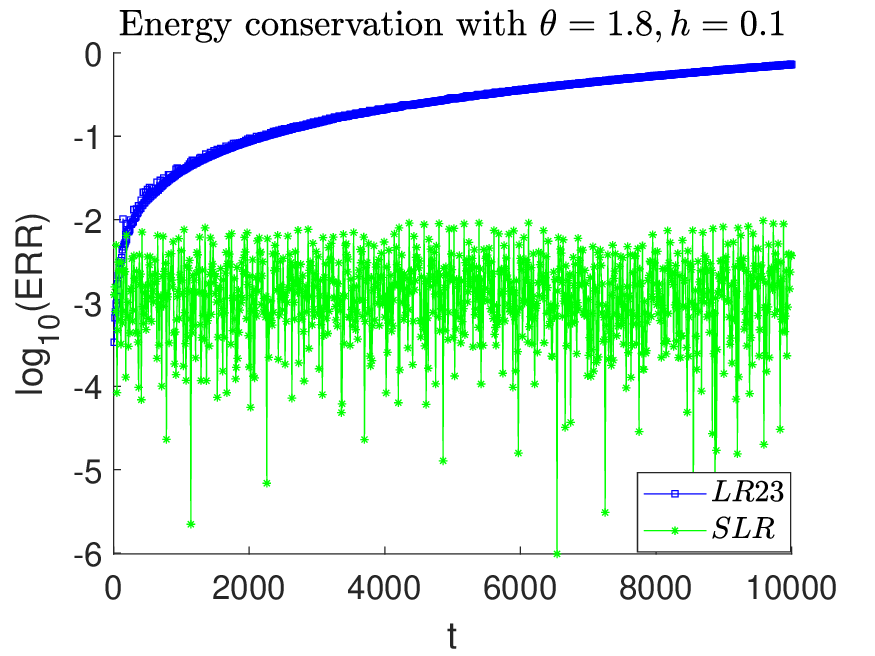}
\caption{Errors of the energy with  initial data $H^{\theta}(\mathbb{T}) \times H^{\theta-1}(\mathbb{T})$ against  time $t$ for different $\theta$ and  $h $.} \label{p3}
\end{figure}
\textbf{Long time conservations.}  Finally,  we display the long time conservations of low-regularity methods LR23 and SLR.
The relative errors of discrete energy (denoted by $ERR$) over a long interval $[0,10000]$
 are shown in Figure \ref{p3} for the system with the  initial data $\varphi_1(x)/4$ and $\varphi_2(x)/4$\footnote{The reason for this choice is that a small initial value is needed in the long time analysis given in Section \ref{sec:4}. }.
According to these numerical results, it is known that the energy is nearly preserved numerically by   our
integrator  SLR  over long times. In contrast, LR23 shows substantial drift in the energy quantity and thus this LR method  does not have long-term performance in the energy conservation. The methods have a similar behaviour of momentum and action conservations and we skip them for brevity.

In the rest parts of this paper, we shall present the  rigorous theoretical analysis for these behaviours.  The error estimates are derived in Section \ref{sec:3} and long-time analysis is  deduced in Section \ref{sec:4}.

\section{Convergence}\label{sec:3}
In this section, we derive the convergence result of the proposed  scheme  \eqref{method} and present its proof.
\begin{mytheo}  \label{symplectic thm} (\textbf{Error estimate}) For the  nonlinear function $f$ of the Klein--Gordon equation \eqref{klein-gordon}, the Lipschitz continuity conditions $\abs{f^{(k)}(w)}\leq C_1$
 are assumed   for $w\in \mathbb{R}$ and $k=1,2,3$.
Under the  regularity condition $(u(0,x), \partial_tu(0,x))\in [H^{1+\frac{d}{4}}(\mathbb{T})\bigcap H^{1}_0(\mathbb{T})]\times H^{\frac{d}{4}}(\mathbb{T}),$
 the numerical result $u_n,v_n$ produced  in Definition \ref{def-se} has the
global error  in the energy space $H^1(\mathbb{T})\times L^2(\mathbb{T})$:
 \begin{equation}\label{error bound}
\begin{aligned}
\max_{0\leq n \leq T/h } \norm{u_n-u(t_n)}_{H^{1}} \leq Ch ^2,\ \ \
\max_{0\leq n \leq T/h } \norm{v_n-v(t_n)}_{L^{2}}\leq Ch ^2,
\end{aligned}
\end{equation}
where $C$ is a positive constant which is independent of the stepsize $h $ and $n$ but depends on $C_1$ and $T$.
\end{mytheo}

To prove the result, we need to study the local truncation errors in a very  rigorously way and make full use of  the structure in NKGE to reduce regularity. Moreover, for a dominant  term  in the global errors,
we take a new way to derive optimal error bound on it.
With these deductions, this theorem is finally proved. 
In the analysis, Lipschitz conditions for the nonlinearity are required  which  is necessary to make the proof go smoothly.

\begin{proof}
To estimate the error $e_n^{u}:=u(t_{n})-u_n, e_n^{v}:=v(t_{n})-v_n$ of the scheme \eqref{method}
for $0\leq n\leq  T /h $,
we shall first consider the local truncation errors which are defined by inserting
the solution of  \eqref{klein-gordon} into \eqref{method}:
 \begin{equation}\label{local errors}
\begin{aligned}
\zeta_n^{u}:=&u(t_{n+1})-2\cos(h \sqrt{\mathcal{A}}) u(t_{n})+u(t_{n-1})-h ^2\sinc(  h \sqrt{\mathcal{A}}) f\big( u(t_{n})\big),\\
\zeta_n^{v}:=&v(t_{n+1})+2h \mathcal{A} \sinc( h \sqrt{\mathcal{A}})u(t_{n})-v(t_{n-1})-h \big(\cos(h \sqrt{\mathcal{A}}) +\sinc(  h \sqrt{\mathcal{A}})\big) f\big( u(t_{n})\big),
\end{aligned}
\end{equation}
where $n=1,2,\ldots, T/h -1$.
For $n = 0$,   the local truncation errors are determined by inserting
the solution of  \eqref{klein-gordon} into the starting value $u_1, v_1$ \eqref{start value}
and the analysis of \cite{LI23} shows
    \begin{equation}\label{F bound10}
\begin{aligned}
\norm{\zeta_0^{u}}_{H^{1}}+ \norm{\zeta_0^{v}} _{L^{2}}
 \lesssim h ^3 \max_{t\in[0,t_1]}\Big(\norm{u(t)}^2_{H^{1+\frac{d}{4}}}+\norm{v(t)}^2_{H^{\frac{d}{4}}}\Big).
\end{aligned}
\end{equation}
Then subtracting the corresponding local error terms \eqref{local errors} from the scheme  \eqref{method}, we find
 \begin{equation}\label{equ errors}
\begin{aligned}
&e_{n+1}^{u}-2\cos(h \sqrt{\mathcal{A}})e_n^{u}+e_{n-1}^{u}=\zeta_n^{u}+\eta_n^{u},\ \
e_{n+1}^{v}+2h \mathcal{A} \sinc( h \sqrt{\mathcal{A}})e_n^{u}-e_{n-1}^{v}=\zeta_n^{v}+\eta_n^{v},
\end{aligned}
\end{equation}
where $n=1,2,\ldots,\frac{T}{h }-1$ and the following notations are used
 \begin{equation}\label{sta errors}
\begin{aligned}
&\eta_n^{u}:=h ^2\sinc(  h \sqrt{\mathcal{A}}) \Big(f( u_n)-f\big( u(t_{n})\big)\Big),\ \
 \eta_n^{v}:=h \big(\cos(h \sqrt{\mathcal{A}}) +\sinc(  h \sqrt{\mathcal{A}})\big) \Big(f( u_n)-f\big( u(t_{n})\big)\Big).
\end{aligned}
\end{equation}
For the starting value, by the choice of $u_0, v_0$ and $u_1, v_1$, we simply have
$e_0^{u}=e_0^{v}=0$ and $e_1^{u}=\zeta_0^{u}, e_1^{v}=\zeta_0^{v}.
$


\subsection{Estimate of the local errors $\zeta_n^{u}, \zeta_n^{v}$}\label{sec:31}
According to the Duhamel’s formulae \eqref{u2add}, the local errors $\zeta_n^{u}, \zeta_n^{v}$ for $n\geq1$ defined in  \eqref{local errors} can be written as
 \begin{equation}\label{local errors 1}
\begin{aligned}
\zeta_n^{u}=&\int_0^{h  } (h -s)\sinc( (h -s)\sqrt{\mathcal{A}}) \Big(f\big( u(t_n+s)\big)+f\big( u(t_n-s)\big)\Big)ds-h ^2\sinc(  h \sqrt{\mathcal{A}}) f\big( u(t_{n})\big),\\
\zeta_n^{v}=&\int_0^{h  }  \cos( (h -s)\sqrt{\mathcal{A}}) \Big(f\big( u(t_n+s)\big)+f\big( u(t_n-s)\big)\Big)ds-h \big(\cos(h \sqrt{\mathcal{A}}) +\sinc(  h \sqrt{\mathcal{A}})\big) f\big( u(t_{n})\big).
\end{aligned}
\end{equation}
In what follows, we rigorously derive the estimate of  $\zeta_n^{u}$ and with the same arguments, the boundedness of $\zeta_n^{v}$ can be deduced.

Utilizing the Duhamel’s formula \eqref{Duhamel}, we get the expression of $u(t_n\pm s)$.
Inserting the result into $f\big(u(t_n\pm s)\big)$ and taking Taylor expansion at $\alpha^{\intercal}(\pm s)U(t_n)$ (we denote $U(t_n)=(
             u(t_n),
             v(t_n) )^{\intercal}$  form now on) yields
\begin{equation}\label{funs}
\begin{aligned}
f\big(u(t_n\pm s)&\big)=f\Big(\alpha^{\intercal}(\pm s)U(t_n)\Big)+f'\Big(\alpha^{\intercal}(\pm s)U(t_n)\Big)
\Big(u(t_n\pm s)-\alpha^{\intercal}(\pm s)U(t_n)\Big)\\
&+R_{f''}(t_n,\pm s)\Big(u(t_n\pm s)-\alpha^{\intercal}(\pm s)U(t_n),u(t_n\pm s)-\alpha^{\intercal}(\pm s)U(t_n)\Big),
\end{aligned}
\end{equation}
where
 $\begin{aligned}&R_{f''}(t_n,\pm s)
:=\int_0^1 \int_0^1\theta f''\Big( (1-\xi)\alpha^{\intercal}(\pm s)U(t_n)+
\xi(1-\theta) \alpha^{\intercal}(\pm s)U(t_n)+\theta u(t_n\pm s)\Big)d\xi d\theta.\end{aligned}$
Replacing $u(t_n\pm s)$ appeared in the right hand side of \eqref{funs} by the Duhamel’s formula \eqref{Duhamel}, we obtain
\begin{equation}\label{funs-new}
\begin{aligned}
&f\big(u(t_n\pm s)\big)=\underbrace{f\Big(\alpha^{\intercal}(\pm s)U(t_n)\Big)}_{=:P_{f}(t_n,\pm s)}
+\underbrace{f'\Big(\alpha^{\intercal}(\pm s)U(t_n)\Big)
 I(t_n,\pm s)}_{=:P_{f'}(t_n,\pm s)}+\widetilde{R}_{f''}(t_n,\pm s),
\end{aligned}
\end{equation}
where $I(t_n,\pm s):=\int_0^{\pm s }  \frac{\sin( (\pm s-\theta)\sqrt{\mathcal{A}})}{\sqrt{\mathcal{A}}} f\big( u(t_n+\theta)\big)d\theta$ and $$\widetilde{R}_{f''}(t_n,\pm s):
= R_{f''}(t_n,\pm s) \big(I(t_n,\pm s),  I(t_n,\pm s)\big).$$
Inserting \eqref{funs-new} into the first formula of \eqref{local errors 1} yields
 \begin{equation}\label{local errors 1-new}
\begin{aligned}
\zeta_n^{u}=&\underbrace{\int_0^{h  } (h -s)\sinc( (h -s)\sqrt{\mathcal{A}}) \Big(P_{f}(t_n,  s)+P_{f}(t_n,- s)\Big)ds
-h ^2\sinc(  h \sqrt{\mathcal{A}}) f\big( u(t_{n})\big)}_{=:\textmd{Part I}^{u}}\\
&+\underbrace{\int_0^{h  } (h -s)\sinc( (h -s)\sqrt{\mathcal{A}}) \Big(P_{f'}(t_n,  s)+P_{f'}(t_n,- s)\Big)ds}_{=:\textmd{Part II}^{u}}
\\
&+\underbrace{\int_0^{h  } (h -s)\sinc( (h -s)\sqrt{\mathcal{A}}) \Big(\widetilde{R}_{f''}(t_n,  s)+\widetilde{R}_{f''}(t_n,- s)\Big)ds}_{=:\textmd{Part III}^{u}}.
\end{aligned}
\end{equation}

 $\bullet$ \textbf{Approximation to Part I$^u$.}
 For the approximation to Part I$^u$ appeared in \eqref{local errors 1-new}, we extract a key part from Part I$^u$ and denote it by $\mathcal{I}(t_n,h ):= \int_0^{h  } (h -s)\sinc( (h -s)\sqrt{\mathcal{A}})  P_{f}(t_n,  s) ds$.
By introducing the function $F(t_n+s):=\beta(-s) f\big( \alpha^{\intercal}(s) U(t_n)\big)$, this part can be formulated as
 \begin{equation}\label{P1-new1}
\begin{aligned}
  \mathcal{I}(t_n,h )=&\int_0^{h  }  \alpha^{\intercal}(h ) \beta(-s) f\big( \alpha^{\intercal}(s) U(t_n)\big) ds=\int_0^{h  }  \alpha^{\intercal}(h ) F(t_n+s) ds\\
 =&\int_0^{h  }  \alpha^{\intercal}(h ) F(t_n) ds+\int_0^{h  }  \alpha^{\intercal}(h )
 F'(t_n+\zeta)(h -\zeta) d\zeta.
\end{aligned}
\end{equation}
Considering the fact that
$
  \alpha^{\intercal}(h )=\alpha^{\intercal}(h -s) M(s)\  \textmd{with}\ M(s):=\left(
                                 \begin{array}{cc}
                                  \cos(s\sqrt{\mathcal{A}}) &  s \sinc( s\sqrt{\mathcal{A}}) \\
                                 - \sqrt{\mathcal{A}} \sin( s\sqrt{\mathcal{A}}) &  \cos(s\sqrt{\mathcal{A}})  \\
                                 \end{array}
                               \right)
$
for any $s\in\mathbb{R},$
\eqref{P1-new1} can be expressed as
 \begin{equation}\label{P1-new2}
\begin{aligned}
 &\mathcal{I}(t_n,h )=\int_0^{h  }  \alpha^{\intercal}(h ) F(t_n) ds+\int_0^{h  }  (h -s) \alpha^{\intercal}(h -2s)
 M(2s)F'(t_n+s) ds\\
 =&\int_0^{h  }  \alpha^{\intercal}(h ) F(t_n) ds+\int_0^{h  }  (h -s) \alpha^{\intercal}(h -2s)\Big(
 M(0)F'(t_n) +\int_0^{s}
 \frac{d  \big(M(2\zeta)F'(t_n+\zeta)\big)}{d\zeta}  d\zeta\Big)ds.
\end{aligned}
\end{equation}
We first compute $F'(t_n+\zeta)$:
 \begin{equation}\label{DF}
\begin{aligned}
&F'(t_n+\zeta)=\frac{d}{d\zeta}  \beta(-\zeta) f\big( \alpha^{\intercal}(\zeta) U(t_n)\big)
    =M(-\zeta)\big(  - f\big( \alpha^{\intercal}(\zeta) U(t_n)\big),
                                   \frac{d}{d\zeta}  f\big( \alpha^{\intercal}(\zeta) U(t_n)\big)
                              \big)^{\intercal}.
\end{aligned}
\end{equation}
Then it is obtained that
$ \frac{d  }{d\zeta} M(2\zeta)F'(t_n+\zeta) =\frac{d  }{d\zeta} M(\zeta)\left(
                                \begin{array}{c}
                                  - f\big( \alpha^{\intercal}(\zeta) U(t_n)\big) \\
                                   \frac{d}{d\zeta}  f\big( \alpha^{\intercal}(\zeta) U(t_n)\big)  \\
                                \end{array}
                              \right)
                             =\beta(\zeta)\Upsilon(t_n,\zeta)
$
with the notation \begin{equation}\label{UPS}
\begin{aligned}&\Upsilon(t_n,\zeta):=   \mathcal{A} f\big( \alpha^{\intercal}(\zeta) U(t_n)\big)
                                 +\frac{d^2}{d\zeta^2}  f\big( \alpha^{\intercal}(\zeta) U(t_n)\big) \\
                                                 =&   \mathcal{A} f\big( \alpha^{\intercal}(\zeta) U(t_n)\big)
                                 + \frac{d }{d\zeta } \Big( f'\big( \alpha^{\intercal}(\zeta) U(t_n)\big) \gamma^{\intercal}(\zeta) U(t_n)\Big)\\
                                 =&   \mathcal{A} f\big( \alpha^{\intercal}(\zeta) U(t_n)\big)
                                 +   f''\big( \alpha^{\intercal}(\zeta) U(t_n)\big)\big(\gamma^{\intercal}(\zeta) U(t_n),\gamma^{\intercal}(\zeta) U(t_n) \big)- f'\big( \alpha^{\intercal}(\zeta) U(t_n)\big)\mathcal{A}\alpha^{\intercal}(\zeta) U(t_n)\\
                                             =&   f'\big( \alpha^{\intercal}(\zeta) U(t_n)\big)\mathcal{A}\alpha^{\intercal}(\zeta) U(t_n)-   f''\big( \alpha^{\intercal}(\zeta) U(t_n)\big)\big(\nabla \alpha^{\intercal}(\zeta) U(t_n),\nabla \alpha^{\intercal}(\zeta) U(t_n) \big)\\
                                 &+   f''\big( \alpha^{\intercal}(\zeta) U(t_n)\big)\big(\gamma^{\intercal}(\zeta) U(t_n),\gamma^{\intercal}(\zeta) U(t_n) \big)- f'\big( \alpha^{\intercal}(\zeta) U(t_n)\big)\mathcal{A}\alpha^{\intercal}(\zeta) U(t_n)\\
                                     =&f''\big( \alpha^{\intercal}(\zeta) U(t_n)\big)\big(\gamma^{\intercal}(\zeta) U(t_n),\gamma^{\intercal}(\zeta) U(t_n) \big)-  f''\big( \alpha^{\intercal}(\zeta) U(t_n)\big)\big(\nabla \alpha^{\intercal}(\zeta) U(t_n),\nabla \alpha^{\intercal}(\zeta) U(t_n) \big). \end{aligned}
\end{equation}
Here we have used the properties $\alpha'(\zeta)=\gamma(\zeta)$ and  $\gamma'(\zeta)= \mathcal{A}\alpha(\zeta)$.
It is noted that the result of $\Upsilon(t_n,\zeta)$ is the key point  for reducing the regularity since it changes $ \mathcal{A}$ into $\nabla $ in the expression.
Now the result of  \eqref{P1-new2} can be simplified as \begin{equation*}
\begin{aligned}
  &\mathcal{I}(t_n,h )=\int_0^{h  }  \alpha^{\intercal}(h )  ds F(t_n) +\int_0^{h  }  (h -s) \alpha^{\intercal}(h -2s)
 M(0)ds F'(t_n) +R_1(t_n,h )\\
                              =&h ^2/2\sinc(  h \sqrt{\mathcal{A}}) f\big( u(t_{n})\big)+\frac{ \sin (  h \sqrt{\mathcal{A}}) -  h \sqrt{\mathcal{A}}\cos(  h \sqrt{\mathcal{A}})}{2 \sqrt{\mathcal{A}}^3} f'\big( u(t_n)\big) v(t_n) +R_1(t_n,h ),\end{aligned}
\end{equation*}
where we use the  formula \eqref{DF} to get $F'(t_n)$   and consider a new notation $$R_1(t_n,h ):=\int_0^{h  }  (h -s) \alpha^{\intercal}(h -2s)\int_0^{s}
\beta(\zeta)\Upsilon(t_n,\zeta)  d\zeta ds.$$
Therefore,  the  Part I$^u$  of   \eqref{local errors 1-new}     becomes
$
\textmd{Part I}^u=\mathcal{I}(t_n,h )+\mathcal{I}(t_n,-h )
-h ^2\sinc(  h \sqrt{\mathcal{A}}) f\big( u(t_{n})\big)
=R_1(t_n,h )+R_1(t_n,-h ).
$
Then concerning  the expression \eqref{notations w} of $ \alpha, \beta$, we get that
$$
\norm{\textmd{Part I}^u}_{H^{1}}\lesssim   h ^3  \max_{s \in[0,h ],\zeta\in[-s,s]}\norm{\frac{\Upsilon(t_n,\zeta)}{\sqrt{\mathcal{A}}} }_{H^{1}}
\lesssim  h ^3  \max_{\zeta \in[-h ,h ]} \norm{ \Upsilon(t_n,\zeta)}_{L^{2}}.$$
According to the initial date, we know that $(u(t,x), \partial_tu(t,x))\in [H^{1+\frac{d}{4}}(\mathbb{T})\bigcap H^{1}_0(\mathbb{T})]\times H^{\frac{d}{4}}(\mathbb{T}).$
From the properties of  the semigroup $M(s)$, it follows that $\alpha^{\intercal}(\zeta) U(t_n) \in H^{1+\frac{d}{4}}$ and $\gamma^{\intercal}(\zeta) U(t_n) \in H^{\frac{d}{4}}$. Therefore, using the expression \eqref{UPS} of $\Upsilon$, one has
  \begin{equation}\label{Part I bound2}
\begin{aligned}
&\norm{\textmd{Part I}^{u}}_{H^{1}}  \lesssim  h ^3  \max_{\zeta \in[-h ,h ]} \norm{ \Upsilon(t_n,\zeta)}_{L^{2}} \\
 \lesssim& h ^3
\max_{\zeta\in[-h ,h ]}\Big( \norm{ \gamma^{\intercal}(\zeta) U(t_n)}^2_{L^{4}}+
 \norm{\nabla \alpha^{\intercal}(\zeta) U(t_n)}^2_{L^{4}}+
 \norm{   \alpha^{\intercal}(\zeta) U(t_n)}_{L^{2}}\Big) \\
 \lesssim &h ^3
\max_{\zeta\in[-h ,h ]}\Big( \norm{ \gamma^{\intercal}(\zeta) U(t_n)}^2_{H^{\frac{d}{4}}}+
 \norm{ \alpha^{\intercal}(\zeta) U(t_n)}^2_{H^{1+\frac{d}{4}}}+
 \norm{   \alpha^{\intercal}(\zeta) U(t_n)}_{L^{2}}\Big)\\
 \lesssim  &h ^3\max_{t\in[t_{n-1},t_{n+1}]}\Big(\norm{u(t)}^2_{H^{1+\frac{d}{4}}}+\norm{v(t)}^2_{H^{\frac{d}{4}}}\Big),
\end{aligned}
\end{equation}
where we utilize the  Sobolev embeddings  $H^{\frac{d}{4}} \hookrightarrow L^{4}$ and $H^{1+\frac{d}{4}} \hookrightarrow W^{1,4}.$


$\bullet$ \textbf{Approximation to Part II$^{u}$.}
 Based on the scheme of Part II$^{u}$ in \eqref{local errors 1-new} and $P_{f'}(t_n,  \pm s)$ in \eqref{funs-new}, it follows that
   \begin{equation*}
\begin{aligned}
&\norm{\textmd{Part II}^{u}}_{H^{1}}\lesssim \norm{ \int_0^{h  } (h -s)\sinc( (h -s)\sqrt{\mathcal{A}}) \Big(P_{f'}(t_n,  s)+P_{f'}(t_n,- s)\Big)ds} _{H^{1}}\\
\lesssim & \sum_{m=s,-s}  \int_0^{h  }\abs{h -m}
 \int_0^{  m }\norm{ (  m-\theta)\sinc( (m-\theta)\sqrt{\mathcal{A}}) f\big( u(t_n+\theta)\big)} _{H^{1}}d\theta dm\\
 \lesssim &  h ^4  \max_{\zeta\in[-h ,h ]} \norm{f\big( u(t_n+\zeta)\big)} _{H^{1}} \lesssim   h ^4  \max_{t\in[t_{n-1},t_{n+1}]} \norm{  u(t) } _{H^{1}}.
\end{aligned}
\end{equation*}

 $\bullet$ \textbf{Approximation to Part III$^{u}$.}
 Finally, we consider the bounds of Part III$^{u}$ presented in \eqref{local errors 1-new}. In the light of the result  $\widetilde{R}_{f''}$, it is easy to see that
$\norm{\textmd{Part III}^{u}}_{H^{1}}\lesssim   h ^6  \max_{t\in[t_{n-1},t_{n+1}]} \norm{  u(t) }^2 _{H^{1}}.
$


With the same deductions,  we can get the result of   $\norm{\zeta_n^{v}} _{L^{2}}$ and in summary,   the local errors $\zeta_n^{u}, \zeta_n^{v}$  are bounded as $
\norm{\zeta_n^{u}}_{H^{1}} +
 \norm{\zeta_n^{v}} _{L^{2}}
 \lesssim h ^3 \max_{t\in[t_{n-1},t_{n+1}]}\Big(\norm{u(t)}^2_{H^{1+\frac{d}{4}}}+\norm{v(t)}^2_{H^{\frac{d}{4}}}\Big).
$

\subsection{Proof of  global errors}\label{sec:32}
We are now in the position to derive the global errors of our proposed integrator.
With the analysis of trigonometric integrators presented  in \cite{Grimm05,Lubich99,Zhao23}, we get the  recurrence relation of $e_{n}^{u}$ from the first equation of \eqref{equ errors}
 \begin{equation}\label{equ u reerrors}
\begin{aligned}
&e_{n+1}^{u}=W_{n}e_1^{u}-W_{n-1}e_0^{u}+ \sum_{k=1}^{n}W_{n-k}(\zeta_{k}^{u}+\eta_{k}^{u}),
\ \ \ n=1,2,\ldots,\frac{T}{h }-1,
\end{aligned}
\end{equation}
where $W_{n}$ is an operator defined analytically  $W_{n}:=\frac{\sin\big( (n+1) h \sqrt{\mathcal{A}}\big)}{\sin( h \sqrt{\mathcal{A}})}.$
Then for  the $e_{n+1}^{v}$, we deduce from the second equation of \eqref{equ errors} that
 \begin{equation}\label{equ v reerrors}
\begin{aligned}
&e_{n+1}^{v}=e_{n-1}^{v}+\zeta_n^{v}+\eta_n^{v}- 2h \mathcal{A} \sinc( h \sqrt{\mathcal{A}})\Big(W_{n-1}e_1^{u}+ \sum_{k=1}^{n-1}W_{n-k-1}(\zeta_k^{u}+\eta_k^{u})\Big)\\
=&e_{n-1}^{v}+\zeta_n^{v}+\eta_n^{v}- 2 \sqrt{\mathcal{A}}\Big(\sin( nh \sqrt{\mathcal{A}})e_1^{u}+ \sum_{k=1}^{n-1} \sin( (n-k)h \sqrt{\mathcal{A}})(\zeta_k^{u}+\eta_k^{u})\Big),
\end{aligned}
\end{equation}
where we use the observation $\sin( h \sqrt{\mathcal{A}})W_{n-1}=\sin( nh \sqrt{\mathcal{A}})$ in the last equation.
We carry out an induction proof on the global error \eqref{error bound} and the following boundedness of the numerical
solution
 \begin{equation}\label{boundedness  num}
\begin{aligned}\norm{u_n}_{H^{1}} \leq 1 +\norm{u}_{L^{\infty}([0,T];H^{1})},\ \ \norm{v_n}_{L^{2}} \leq 1 +\norm{v}_{L^{\infty}([0,T];L^{2})}.\end{aligned}
\end{equation}

 Firstly, the statements \eqref{error bound} and \eqref{boundedness  num} are clearly true since
$e_{0}^{u}=e_{0}^{v}=0$. For $n =1$,
by the choice of the initial value $u_1,v_1$ from  \cite{LI23}, we have
$
\norm{e_1^{u}}_{H^{1}}\lesssim     h ^3 ,   \norm{e_1^{v}} _{L^{2}} \lesssim h ^3$ and thus \eqref{error bound}  is true for $n=1$. These results give further that
  there exists some constant $h _0 > 0$ such that when $h < h _0$, \eqref{boundedness  num} holds for $n=1$.
  Now assume
that \eqref{error bound} and \eqref{boundedness  num}  hold for $1 \leq n\leq m < T /h $, and we are going to prove the validity at $m + 1$.

Taking the $H^{1}$
-norm on both sides of  \eqref{equ u reerrors}, we get
$$\norm{e_{n+1}^{u}}_{H^{1}}\leq  (n+1)\norm{e_1^{u}}_{H^{1}}+ \sum_{k=1}^{n} (n-k+1)(\norm{\zeta_k^{u}}_{H^{1}}+\norm{\eta_k^{u}}_{H^{1}}),
$$
where we have used the property $\norm{W_{n}g}_{H^{r}}\leq (n+1) \norm{g}_{H^{r}}$   for some function $g(x)$ on $\mathbb{T}$.
Noticing the scheme of  \eqref{sta errors}, we find that $\eta_n^{u}, \eta_n^{v}$ are bounded by
$\norm{\eta_n^{u}}_{H^{1}}\lesssim h ^2 \norm{e_n^{u}}_{H^{1}},
\norm{\eta_n^{v}}_{H^{1}}\lesssim h  \norm{e_n^{u}}_{H^{1}}.$
Based on them and those local errors shown in the previous subsection, for $n=m$ we obtain
$$\norm{e_{m+1}^{u}}_{H^{1}}\lesssim  (m+1)h ^3+ \sum_{k=1}^{m} (m-k+1)(h ^3+h ^2\norm{e_k^{u}}_{H^{1}})
\lesssim  h + h   \sum_{k=1}^{m}  \norm{e_k^{u}}_{H^{1}},
$$
which by Gronwall’s inequality shows a  coarse result $\norm{e_{m+1}^{u}}_{H^{1}}\lesssim h.$
Thus   when $h $ is small enough,
\eqref{boundedness  num} holds for $u^{n}$ with $n = m + 1$.

To get  an optimal convergence of $e_{m+1}^{u}$,   we deal with \eqref{equ u reerrors} by
 \begin{equation}\label{equ u reerrors-12}
\begin{aligned}
&\norm{e_{m+1}^{u}}_{H^{1}}\leq  (m+1)\norm{e_1^{u}}_{H^{1}}+ \norm{\sum_{k=1}^{m} W_{m-k}\zeta_k^{u}}_{H^{1}}+\sum_{k=1}^{m} (m-k+1)\norm{\eta_k^{u}}_{H^{1}},
\end{aligned}
\end{equation} and deduce a fine estimate of $\sum_{k=1}^{m} W_{m-k}\zeta_k^{u}$.
This part can be rewritten by   partial summation:
$$
\sum_{k=1}^{m}W_{m-k} \zeta_k^{u}=\sum_{j=1}^{m}W_{j} \zeta_{m+1-j}^{u}
=\sum_{j=1}^{m}\frac{ \sin( jh \sqrt{\mathcal{A}}) }{\sin( h \sqrt{\mathcal{A}})} \zeta_{m+1-j}^{u}
=E_{m}\zeta_{1}^{u}+\sum_{j=1}^{m-1}E_j\big(\zeta_{m-j}^{u}-\zeta_{m+1-j}^{u}\big),
$$
where $E_{j}:=\sum_{k=1}^{j}\frac{\sin( kh \sqrt{\mathcal{A}})}{\sin( h \sqrt{\mathcal{A}})}=\frac{\sin(  \frac{j+1}{2}h \sqrt{\mathcal{A}})
\sin(  \frac{j}{2}h \sqrt{\mathcal{A}}) }{\sin(  \frac{1}{2}h \sqrt{\mathcal{A}}) \sin( h \sqrt{\mathcal{A}})}.$
Meanwhile, the dominating term  in the bound of $\zeta_{m-j}^{u}$ comes from  Part I$^u$ appeared in \eqref{local errors 1-new} since the others are at least $\mathcal{O}(h ^4)$. Therefore, we obtain
 \begin{equation}\label{p s-1}
\begin{aligned}
 &\norm{\sum_{k=1}^{m}W_{m-k} \zeta_k^{u}}_{H^1}\lesssim  \norm{E_{m}\zeta_{1}^{u}}_{H^1}+
 \sum_{j=1}^{m-1} \norm{E_j\zeta_{m-j+1}^{u}-E_j\zeta_{m-j}^{u}}_{H^1}
 \lesssim \norm{E_{m}(R_1(t_{1},h )+R_1(t_{1},-h ))}_{H^1}\\+&
 \sum_{j=1}^{m-1} \norm{E_j R_1(t_{m-j+1},h )-E_j R_1(t_{m-j},h )}_{H^1} + \sum_{j=1}^{m-1} \norm{E_j R_1(t_{m-j+1},-h ) -E_j R_1(t_{m-j},-h )}_{H^1}+C h ^2.
\end{aligned}
\end{equation}
We now consider
 \begin{equation*}
\begin{aligned}
 & \norm{E_j R_1(t_{m-j+1},h )-E_j R_1(t_{m-j},h )}_{H^1}\\
\lesssim& h ^3  \max_{s,\zeta \in[0,h ]}\norm{\Big( \frac{ \sin(\zeta \sqrt{\mathcal{A}})}{\sin(  \frac{1}{2}h \sqrt{\mathcal{A}}) } +\frac{  \sin( (h -2s)\sqrt{\mathcal{A}})}{\sin(  \frac{1}{2}h \sqrt{\mathcal{A}})  }  \Big)\sin(  \frac{j+1}{2}h \sqrt{\mathcal{A}})
\sin(  \frac{j}{2}h \sqrt{\mathcal{A}})  }_{L^{2}}\\& \max_{s,\zeta \in[0,h ]}\norm{ \frac{1}{\sin( h \sqrt{\mathcal{A}})}\big(  \Upsilon(t_{m-j+1},\zeta)- \Upsilon(t_{m-j},\zeta) \big)}_{L^{2}}\\
\lesssim&   h ^3 \Big( \max_{\zeta \in[0,h ]}\abs{\frac{\zeta}{h /2}}+ \max_{s \in[0,h ]}\abs{\frac{h -2s}{h /2}}\Big) \norm{\frac{ \Upsilon(t_{m-j+1},\zeta)- \Upsilon(t_{m-j},\zeta) }{  h \sqrt{\mathcal{A}} }   }_{L^{2}}\\
\lesssim&    h ^3  \max_{\zeta \in[0,h ]}\norm{ \frac{ \Upsilon(t_{m-j+1},\zeta)- \Upsilon(t_{m-j},\zeta) }{  h \sqrt{\mathcal{A}} }  }_{L^{2}}.
\end{aligned}
\end{equation*}
Keeping the expression of $\Upsilon$ \eqref{UPS} in mind,
it is deduced that
\begin{equation*}
\begin{aligned}
  &\Upsilon(t_{m-j},\zeta)- \Upsilon(t_{m-j-1},\zeta)  \\
=&
f''\big( \alpha^{\intercal}(\zeta) U(t_{m-j})\big)\big(\gamma^{\intercal}(\zeta) U(t_{m-j}),\gamma^{\intercal}(\zeta) U(t_{m-j}) \big)\\&-  f''\big( \alpha^{\intercal}(\zeta) U(t_{m-j})\big)\big(\nabla \alpha^{\intercal}(\zeta) U(t_{m-j}),\nabla \alpha^{\intercal}(\zeta) U(t_{m-j}) \big)\\
&-f''\big( \alpha^{\intercal}(\zeta) U(t_{m-j-1})\big)\big(\gamma^{\intercal}(\zeta) U(t_{m-j-1}),\gamma^{\intercal}(\zeta) U(t_{m-j-1}) \big)\\&+  f''\big( \alpha^{\intercal}(\zeta) U(t_{m-j-1})\big)\big(\nabla \alpha^{\intercal}(\zeta) U(t_{m-j-1}),\nabla \alpha^{\intercal}(\zeta) U(t_{m-j-1}) \big).
\end{aligned}
\end{equation*}
Therefore, it follows that
\begin{equation*}
\begin{aligned}
&\norm{  \Upsilon(t_{m-j},\zeta)- \Upsilon(t_{m-j-1},\zeta) }_{L^{2}}\\
  \lesssim &   \norm{\frac{ \gamma^{\intercal}(\zeta)}{ h \sqrt{\mathcal{A}}} \big(U(t_{m-j})-U(t_{m-j-1})\big)}^2_{L^{4}}+
 \norm{\frac{\nabla \alpha^{\intercal}(\zeta)}{h \sqrt{\mathcal{A}}}\big(U(t_{m-j})-U(t_{m-j-1})\big)}^2_{L^{4}}\\&+
 \norm{  \frac{ \alpha^{\intercal}(\zeta) }{h \sqrt{\mathcal{A}}}\big(U(t_{m-j})-U(t_{m-j-1})\big)}_{L^{2}} \\
  \lesssim &
\norm{\frac{u(t_{m-j})-u(t_{m-j-1})}{h \sqrt{\mathcal{A}}}}^2_{H^{1+\frac{d}{4}}}+\norm{\frac{v(t_{m-j})-v(t_{m-j-1})}{h \sqrt{\mathcal{A}}}}^2_{H^{\frac{d}{4}}} \lesssim 1.
\end{aligned}
\end{equation*}
Thus we get $$\sum_{j=1}^{m-1} \norm{E_j R_1(t_{m-j+1},h ) -E_j R_1(t_{m-j},h )}_{H^1}\lesssim  \sum_{j=1}^{m-1}  h ^3 \lesssim h^2.$$
By the same arguments, we deduce that
\begin{equation*}
\begin{aligned}
&\sum_{j=1}^{m-1} \norm{E_j R_1(t_{m-j+1},-h ) -E_j R_1(t_{m-j},-h )}_{H^1}\lesssim  h^2,\ \ \ \norm{E_{m}(R_1(t_{1},h )+R_1(t_{1},-h ))}_{H^1}\lesssim  h^2.
\end{aligned}
\end{equation*}
Therefore, \eqref{p s-1} is bounded by $\norm{\sum_{k=1}^{m}W_{m-k} \zeta_k^{u}}_{H^1}\lesssim  h^2$ and then \eqref{equ u reerrors-12} becomes
 \begin{equation*}
\begin{aligned}
&\norm{e_{m+1}^{u}}_{H^{1}}
\lesssim  (m+1)h ^3+ h^2+\sum_{k=1}^{m} (m-k+1) h ^2\norm{e_k^{u}}_{H^{1}}
\lesssim  h^2 + h   \sum_{k=1}^{m}  \norm{e_k^{u}}_{H^{1}}.
\end{aligned}
\end{equation*}
By Gronwall’s inequality,  it is immediately obtained that $\norm{e_{m+1}^{u}}_{H^{1}}\lesssim h^2,$ which proves the statement \eqref{error bound} for $e_{n}^{u}$ with $n = m + 1$.

Then for $e_{m+1}^{v}$, \eqref{equ v reerrors} leads to
 \begin{equation}\label{equ v reerrors-1}
\begin{aligned}
 \norm{e_{m+1}^{v}}_{L^{2}}\lesssim & \norm{e_{m-1}^{v}}_{L^{2}}+\norm{\zeta_m^{v}}_{L^{2}}+\norm{\eta_m^{v}}_{L^{2}}+ \norm{\sqrt{\mathcal{A}} e_1^{u}}_{L^{2}}+ \sum_{k=1}^{m-1}\norm{\sqrt{\mathcal{A}}\sin( (m-k)h \sqrt{\mathcal{A}})(\zeta_k^{u}+\eta_k^{u})}_{L^{2}} \\
\lesssim &\norm{e_{m-1}^{v}}_{L^{2}}+h ^3+\norm{  e_1^{u}}_{H^{1}}+ \sum_{k=1}^{m-1}(h ^3+h ^4) \lesssim \norm{e_{m-1}^{v}}_{L^{2}}+h ^2.\\
\end{aligned}
\end{equation}
This shows a  coarse estimate for $e_{m+1}^{v}:$ $\norm{e_{m+1}^{v}}_{L^{2}}\lesssim h.$  Under the condition that $h $ is small enough,
\eqref{boundedness  num} holds for $n = m + 1$, and so the induction for \eqref{boundedness  num} is done.
In what follows,  we give the proof of the   global error  $e_{m+1}^{v}$,  which refines the error bound to a second order accuracy in time. According to the  procedure \eqref{equ v reerrors-1}, the constraint of  the accuracy comes from
$\sum_{k=1}^{n-1} \sin( (m-k)h \sqrt{\mathcal{A}}) \zeta_k^{u}$. With the same analysis presented above,
this dominating part can be bounded by
 \begin{equation*}
\begin{aligned}
\norm{\sum_{k=1}^{m-1} \sin( (m-k)h \sqrt{\mathcal{A}}) \zeta_k^{u}}_{H^1}
\lesssim & h ^3+ h ^3  \sum_{j=1}^{m-2}  \max_{\zeta \in[-h,h ]}\norm{  \frac{ \Upsilon(t_{m-j},\zeta)- \Upsilon(t_{m-j-1},\zeta)}{h \sqrt{\mathcal{A}}}  }_{L^{2}}\lesssim h ^3.
\end{aligned}
\end{equation*}
Therefore, we have
 \begin{equation*}
\begin{aligned}
 \norm{e_{m+1}^{v}}_{L^{2}}\lesssim & \norm{e_{m-1}^{v}}_{L^{2}}+\norm{\zeta_m^{v}}_{L^{2}}+\norm{\eta_m^{v}}_{L^{2}}+ \norm{\sqrt{\mathcal{A}} e_1^{u}}_{L^{2}}+ \sum_{k=1}^{m-1}\norm{\sin( (n-k)h \sqrt{\mathcal{A}})(\zeta_k^{u}+\eta_k^{u})}_{H^{1}}
 \\
\lesssim &\norm{e_{m-1}^{v}}_{L^{2}}+h ^3+\norm{  e_1^{u}}_{H^{1}}+\sum_{k=1}^{m-1} \norm{\eta_k^{u}}_{H^{1}} \lesssim \norm{e_{m-1}^{v}}_{L^{2}}+h ^3,
\end{aligned}
\end{equation*}
which  by Gronwall’s inequality shows $\norm{e_{m+1}^{v}}_{L^{2}}\lesssim h ^2.$
This means that the statement \eqref{error bound} holds for $e_{n}^{v}$ with $n = m + 1$, and so the induction for \eqref{error bound} is finished.

The proof of the convergence is complete.
\end{proof}
 \section{Long term analysis}\label{sec:4}
In this section, in order to make the expression be
concise, we consider  one-dimensional Klein-Gordon \eqref{klein-gordon} with $d=1$ and $\mathbb{T}=[-\pi,\pi]$.
It is noted that the content of this section can be presented for two or three dimensional Klein-Gordon equation without any difficulty (only the notations and expressions will become complex and tedious).

\subsection{Main result  of this section}
The long term analysis will be given for the fully discrete scheme which considers the  pseudo-spectral semi-discretisation (\cite{Shen}) in space  $x\in \mathbb{T}$.
Choose a   positive integer $M$ and divide $[-\pi,\pi]$ into a series of equidistant small intervals with the  equidistant collocation points $x_k = k\pi
/M$ for $k \in \{-M,-M+1,\ldots,M-1\}$. Consider the
following real-valued trigonometric polynomial as an approximation
for the solution $ u(t,x)$ of  \eqref{klein-gordon}
\begin{equation}\label{trigo pol}
\begin{array}[c]{ll}
 u^{M}(t,x)=\sum\limits_{|j|\leq M}^{'}
 q_j(t)\mathrm{e}^{\mathrm{i}jx}\approx  u(t,x),
\end{array}
\end{equation}
where  $q_j$ are the discrete Fourier coefficients of $u^{M}(t,x)$ and the prime indicates that
the first and last terms in the summation are taken with the factor
$1/2$. Inserting  $u^{M}(t,x)$ into  \eqref{klein-gordon} and collecting all coefficients $q_j$ in a $(2M+1)$-dimensional coefficient vector $$\boldsymbol{q} =
(q_{M},q_{M-1},\ldots,q_{1},q_0,q_{-1},\ldots,q_{-M+1},q_{-M})^\intercal\footnote{In this section, we use the bold symbols for all the $(2M+1)$-dimensional vectors and their elements are denoted in the same way as $\boldsymbol{q}$.},$$
we get a  system of ODEs
\begin{equation}
\frac{d^2 \boldsymbol{q}}{dt^2}+\Omega^2 \boldsymbol{q}=\tilde{f}(\boldsymbol{q}),\ \ \boldsymbol{q}(0)=\boldsymbol{q}_0:=\mathcal{F}\varphi_1(x),\ \ \frac{d }{dt}\boldsymbol{q}(0)=\boldsymbol{p}_0:=\mathcal{F}\varphi_2(x),\label{prob}%
\end{equation}
where $\Omega$ is diagonal with entries $\omega_j=\frac{1}{\abs{j}}$ for $0<|j|\leq M$ and $0$ for $j=0$, and $\mathcal{F}$ denotes the discrete Fourier transform
$(\mathcal{F}w)_j=\frac{1}{2M}\sum\limits_{k=-M}^{M-1}w_k\mathrm{e}^{-\mathrm{i}jx_k}$
 for $|j|\leq M$. The nonlinearity $\tilde{f}(\boldsymbol{q})=\mathcal{F}f(\mathcal{F}^{-1}\boldsymbol{q})$  and it is composed of $\tilde{f}_j$ which has the form
$\tilde{f}_j(\boldsymbol{q})=-\frac{\partial}{\partial q_{-j}}V(\boldsymbol{q})$  with
$V(\boldsymbol{q})=\frac{1}{2M}\sum\limits_{k=-M}^{M}U((\mathcal{F}^{-1}\boldsymbol{q})_k).$
The system \eqref{prob} can be reformulated as a finite-dimensional complex
Hamiltonian system $\frac{d \boldsymbol{q}}{dt}=\nabla_{\boldsymbol{p}}H_M(\boldsymbol{q},\boldsymbol{p}),\ \frac{d \boldsymbol{p}}{dt}=-\nabla_{\boldsymbol{q}}H_M(\boldsymbol{q},\boldsymbol{p})$ with the energy $H_M(\boldsymbol{q},\boldsymbol{p})=\frac{1}{
2}\sum\limits_{|j|\leq M}^{'}\big(
|p_j|^2+\omega_j^2|q_j|^2\big)+V(\boldsymbol{q}).$ The actions (for $|j|\leq M$)
and the momentum of \eqref{prob} are respectively given by $
I_j(q,p)=\frac{\omega_j}{ 2}|q_j|^2+ \frac{1}{ 2\omega_j}|p_j|^2$ and $
K(q,p)=-\sum\limits_{|j|\leq M}''\mathrm{i}jq_{-j}p_{j},$
  where the double prime indicates that the first and
last terms in the summation are taken with the factor $1/4$.
This section is interested in real approximation \eqref{trigo pol} and thus
we have
  $q_{-j} = \bar{q}_j$  and $p_{-j} =
\bar{p}_{j}$  for $|j|\leq M$. Applying  the low-regularity (LR) integrator proposed in Definition \ref{def-se} to the system \eqref{prob} leads to the following fully discrete LR  integrator.

\begin{mydef} (\textbf{Fully discrete LR integrator})
Choose a   positive integer $M$ for the  pseudo-spectral method and  a time stepsize $0<h <1$. Denote the numerical solution as  $\boldsymbol{q}_{n}\approx \boldsymbol{q}(t_{n})$ and   $\boldsymbol{p}_{n}\approx \boldsymbol{p}(t_{n})$ at  $t_n=nh $. The scheme for solving  \eqref{prob} is defined as
 \begin{equation}\label{methodF}
\begin{aligned}
\boldsymbol{q}_{n+1}=&2\cos(h \Omega) \boldsymbol{q}_n-\boldsymbol{q}_{n-1}+h ^2\sinc(  h \Omega) \tilde{f}( \boldsymbol{q}_n),\\
\boldsymbol{p}_{n+1}=&-2h \Omega^2 \sinc( h \Omega) \boldsymbol{q}_n+\boldsymbol{p}_{n-1}+h \big(\cos(h \Omega) +\sinc(  h \Omega)\big) \tilde{f}( \boldsymbol{q}_n),
\end{aligned}
\end{equation}
where  $ n=1,2,\ldots,T/h -1$.
The starting value $\boldsymbol{q}_1,\boldsymbol{p}_1$ is given by
the fully discrete scheme of   \eqref{start value}.
\end{mydef}

In this section, for
 \begin{equation*}
\begin{aligned}
\boldsymbol{k}=&
(k_{M},k_{M-1},\ldots,k_{1},k_0,k_{-1},\ldots,k_{-M+1},k_{-M})^\intercal,\\
\boldsymbol{\omega} =&
(\omega_{M},\omega_{M-1},\ldots,\omega_{1},\omega_0,\omega_{-1},\ldots,\omega_{-M+1},\omega_{-M})^\intercal,
\end{aligned}
\end{equation*}
    we denote
\begin{equation}\label{denot}
\begin{aligned}
&|\boldsymbol{k}| = (|k_l|)_{\abs{l}\leq M},\ \
\norm{\boldsymbol{k}}=\sum\limits_{\abs{l}\leq M}|k_l|, \  \ \boldsymbol{k}\cdot
\boldsymbol{\omega} =\sum\limits_{\abs{l}\leq M}k_l\omega_l, \ \
 \boldsymbol{\omega} ^{\sigma |\boldsymbol{k}| }=\Pi_{\abs{l}\leq M} \omega_l^{\sigma |k_l|}\
\textmd{for real}\ \sigma.
\end{aligned}
\end{equation}
   The vector $(0, \ldots , 0, 1, 0, \ldots,0)^\intercal$ with
the only entry at the $|j|$-th position for $j\in \{-M,-M+1,\ldots,M\}$ is
denoted by  $\langle j\rangle$.  The weighted norm of the Sobolev space $H^{s}$ is denoted in this section that
$\norm{\boldsymbol{q}}_{H^s}=\Big(\sum\limits_{|j|\leq M}''\omega_j^{2s}
|q_j|^2\Big)^{1/2}$ with $s=0,1$.
For the long term analysis,  we need the following assumptions which  have
been considered in \cite{Cohen08-1}.
\begin{assum}\label{ass}
 $\bullet$
The initial values $\boldsymbol{q}(0)$ and $\boldsymbol{p}(0)$ of \eqref{prob} are required to
have a small upper bound $
\big(\norm{\boldsymbol{q}(0)}_{H^{1+s}}^2+\norm{\boldsymbol{p}(0)}_{H^{s}}^2\big)^{1/2}\leq \epsilon
$
with  $0<\epsilon<1$ and $s\geq 1/2$. The nonlinear  function $f(u)$ is assumed to be smooth with $f'(0)=0$.

 $\bullet$ In the proof of the main result, we need a lower bound of the following part
\begin{equation}
|\sin(\frac{h }{2}(\omega_j-\boldsymbol{k}\cdot \boldsymbol{\omega}))\cdot
\sin(\frac{h }{2}(\omega_j+\boldsymbol{k}\cdot \boldsymbol{\omega}))| \geq \epsilon^{1/2}h ^2(
 \omega_j+|\boldsymbol{k}\cdot\boldsymbol{\omega} |),
 \label{inequa}%
\end{equation}
which is called as  the non-resonance condition. If $\boldsymbol{k} \neq\pm\langle j\rangle$ but $(j,\boldsymbol{k})$ satisfies
this condition, we collect them in   a  set
\begin{equation}
\mathcal{S}_{\epsilon,h }=\{(j,\boldsymbol{k}):|j|\leq M,\ \norm{\boldsymbol{k}}\leq2N,\ \
\boldsymbol{k} \neq\pm\langle j\rangle,\ \textmd{satisfying} \
\eqref{inequa}\},
 \label{near-resonant S}%
\end{equation}
where   $N \geq 3$ refers to   the truncation number of the expansion
\eqref{MFE-trigonometric} which will be given below. Otherwise, we put them in a  set of
near-resonant indices
\begin{equation}
\mathcal{R}_{\epsilon,h }=\{(j, \boldsymbol{k}):|j|\leq M,\ \norm{\boldsymbol{k}}\leq2N,\ \
\boldsymbol{k} \neq\pm\langle j\rangle,\ \textmd{not}\ \textmd{satisfying} \
\eqref{inequa}\}.
 \label{near-resonant R}%
\end{equation}
For the elements in this set, we need to further require  that
 there exist  $\sigma
> 0$ and a constant $C_0$ such that
\begin{equation}
\sup_{(j, \boldsymbol{k})\in\mathcal{R}_{\epsilon,h }}
\frac{\omega_j^{\sigma}}{\boldsymbol{\omega} ^{\sigma
|\boldsymbol{k}| }}\epsilon^{\norm{\boldsymbol{k}}/2}\leq C_0 \epsilon^N.
 \label{non-resonance cond}%
\end{equation}
For the  rationality and feasibility of this assumption  \eqref{non-resonance cond}, we refer to Section 2.2 of \cite{Cohen08}.

 $\bullet$
We also assume that
$|\sin(h \omega_j)|\geq h  \epsilon^{1/2}$  for $|j|\leq M,$
and  the following   condition is further required for  a positive constant $c > 0$, $ (j, \boldsymbol{k})$ of the form $j=j_1+j_2$
 and $\boldsymbol{k}=\pm\langle j_1\rangle\pm\langle j_2\rangle$:
\begin{equation}
\begin{aligned}
|\sin&(\frac{h }{2}(\omega_j-\boldsymbol{k}\cdot \boldsymbol{\omega}))\cdot
\sin(\frac{h }{2}(\omega_j+\boldsymbol{k}\cdot \boldsymbol{\omega}))| \geq c h ^2 |\sinc
(h \omega_j)|.
 \end{aligned}
 \label{another-non-res cond}%
\end{equation}

\end{assum}

The main result of this section is  given in  the following theorem.
\begin{mytheo}\label{main theo1} \textbf{(Long-time near conservations)} Under the conditions  of Assumption \ref{ass} and   the requirement  $f(0)=0$,  the fully discrete LR integrator \eqref{methodF} has the following long-time numerical   conservations
\begin{equation}
\begin{aligned}
  &\frac{|H_M(\boldsymbol{q}_n,\boldsymbol{p}_n)-H_M(\boldsymbol{q}_0,\boldsymbol{p}_0)|}{\epsilon^2} \leq C (\epsilon+h ^3),\ \
  \sum\limits_{l=0}^{M}\omega_l^{2s+1}\frac{|I_l(\boldsymbol{q}_n,\boldsymbol{p}_n)I_l
  (\boldsymbol{q}_0,\boldsymbol{p}_0)|}{\epsilon^2}
\leq
C  \epsilon,\\
&\frac{|K(\boldsymbol{q}_n,\boldsymbol{p}_n)-K  (\boldsymbol{q}_0,\boldsymbol{p}_0)|}{\epsilon^2} \leq C
(\epsilon+M^{-s}+\epsilon t M^{-s+1}),
\end{aligned}
\label{near-conser res}%
\end{equation}
 where    $0\leq t=nh \leq
\epsilon^{-N+1}$ and the constant $C$ is independent of   $\epsilon, M, h , n$ but
depends on $N, C_0, c$ appeared in   Assumption \ref{ass}.
\end{mytheo}
\subsection{Proof of Theorem \ref{main theo1}}
Although the technique for the proof is based on  a   multi-frequency modulated Fourier expansion \cite{Cohen08,Cohen08-1,Cohen08-0,W23}, there are some new difficulties in the analysis. The main issue comes from that
the starting value $\boldsymbol{q}_1,\boldsymbol{p}_1$   is given by a non-symmetric method
and this method \eqref{start value}  is not the one-step form of our proposed scheme  \eqref{methodF}.
Therefore, the technique of modulated Fourier expansion cannot be used on the whole interval $0\leq t=nh \leq
\epsilon^{-N+1}$. To solve this problem, we divide the analysis into two parts. The first one is on $[0,h]$ and the error of energy conservation is immediately obtained based on the accuracy of  \eqref{start value}, i.e., $\mathcal{O}(h ^3)$ in \eqref{near-conser res}.
 The second concern is on the long time interval $[h,\epsilon^{-N+1}]$ with the initial value  $
\big(\norm{\boldsymbol{q}_1}_{H^{1+s}}^2+\norm{\boldsymbol{p}_1}_{H^{s}}^2\big)^{1/2}\leq \epsilon
$ and the modulated Fourier expansion is considered for our proposed two-step scheme  \eqref{methodF} to derive the long time conservations. It is noted that the symmetry of the proposed scheme plays   a crucial role in the analysis and it ensures the long time near conservations

\subsubsection{The result of modulated Fourier expansion} \label{subsec:MEF}
\begin{mylemma}\label{MFE lem}
Under the conditions of Theorem \ref{main theo1}, the numerical
solution $(\boldsymbol{q}_n, \boldsymbol{p}_n)$ with $n\geq1$  produced  by
\eqref{methodF} can be expressed by a kind of multi-frequency expansion (called as modulated
Fourier expansion) at $t=nh \geq h $:
\begin{equation}
\begin{aligned} & \boldsymbol{\tilde{q}}(t)= \sum\limits_{\norm{\boldsymbol{k}}\leq 2N} \mathrm{e}^{\mathrm{i}(\boldsymbol{k}\cdot\boldsymbol{\omega}) t}\boldsymbol{\zeta}^{\boldsymbol{k}}(\epsilon t),\
\ \boldsymbol{\tilde{p}}(t)=  \sum\limits_{\norm{\boldsymbol{k}}\leq2N} \mathrm{e}^{\mathrm{i}(\boldsymbol{k}\cdot\boldsymbol{\omega}) t}\boldsymbol{\eta}^{\boldsymbol{k}}(\epsilon t)\\
\end{aligned}
\label{MFE-trigonometric}%
\end{equation}
such that
\begin{equation}
\norm{\boldsymbol{q}_n-\boldsymbol{\tilde{q}}(t)}_{H^{1+s}}\leq
C\epsilon^N,\ \ \ \norm{\boldsymbol{p}_n-\boldsymbol{\tilde{p}}(t)}_{H^{s}}\leq
C\epsilon^{N-1}\quad for \ \ h \leq  t=nh \leq \epsilon^{-1},
\label{error MFE}%
\end{equation}
where we use the notations introduced in \eqref{denot},  $N>0$ refers to   the truncation number of the expansion, and $\boldsymbol{\zeta}^{\boldsymbol{k}}(\tilde{t}), \boldsymbol{\eta}^{\boldsymbol{k}}(\tilde{t})$ are modulation functions
that  vary in a slow time $\tilde{t}:= \epsilon
t$. The expansion \eqref{MFE-trigonometric} is     bounded in the energy space   for  $h \leq  t\leq \epsilon^{-1}$:
\begin{equation}
\norm{\boldsymbol{\tilde{q}}(t)}_{H^{1+s}}+\norm{\boldsymbol{\tilde{p}}(t)}_{H^{s}}\leq C\epsilon.
\label{bound TMFE}%
\end{equation}
Moreover, denoting $r_j:=\tilde{q}_j(t)-\zeta_j^{\langle j\rangle}(\epsilon t)
\mathrm{e}^{\mathrm{i} \omega_j t}-\zeta_j^{-\langle
j\rangle}(\epsilon t) \mathrm{e}^{-\mathrm{i} \omega_j t}$ for $ |j|\leq M,$
the following bound can be derived
\begin{equation}\begin{aligned}
 \norm{\boldsymbol{r}}_{H^{1+s}}\leq C\epsilon^2.
\end{aligned}
\label{jth TMFE}
\end{equation}
The modulation functions $\boldsymbol{\zeta}^{\boldsymbol{k}}$ (as well as their any fixed number of derivatives w.r.t the slow time  $\tilde{t}$)  are bounded by
\begin{equation}\begin{aligned}
&\sum\limits_{\norm{\boldsymbol{k}}\leq2N}\Big(\frac{\boldsymbol{\omega}^{|\boldsymbol{k}| }}{\epsilon^{[[\boldsymbol{k}]]}}\norm{\boldsymbol{\zeta}^{\boldsymbol{k}}(\epsilon
t)}_{H^{s}}\Big)^2\leq C \ \ \ \textmd{with}\ \ \ [[\boldsymbol{k}]]=\left\{\begin{aligned} & (\norm{\boldsymbol{k}}+1)/2,\quad
\boldsymbol{k} \neq0,\\
&3/2,\qquad\quad\quad  \ \   \boldsymbol{k}=0.
\end{aligned}\right.
\end{aligned}
\label{bound modu func}%
\end{equation}
For the functions $\boldsymbol{\eta}^{\boldsymbol{k}}$, they have the connection \eqref{MFE-zetaeta} with   $\boldsymbol{\zeta}^{\boldsymbol{k}}$. All the modulation functions have the properties
$\zeta_{-j}^{-\boldsymbol{k}}=\bar{\zeta_{j}^{\boldsymbol{k}}}$ and $\eta_{-j}^{-\boldsymbol{k}}=\bar{\eta_{j}^{\boldsymbol{k}}}$. The constants $C$ in this lemma are
independent of $\epsilon, M, h $ and   $t \leq \epsilon^{-1}.$
\end{mylemma}
\begin{proof}
The proof of this lemma   will
be composed of three parts.
%
%

 \textbf{Step 1. Modulation equations of the  modulation functions.}

According to the  formulae of  \eqref{methodF}, it is
required that
 \begin{equation}\label{methodpq rea}
\begin{aligned}
&\boldsymbol{\tilde{p}}(t+h )-\boldsymbol{\tilde{p}}(t-h )
=(I+h \Omega\cot(h \Omega))\frac{\boldsymbol{\tilde{q}}(t+h )+\boldsymbol{\tilde{q}}(t-h )}{h }
-2(\cos(h \Omega)+h \Omega\csc(h \Omega))\frac{ \boldsymbol{\tilde{q}}(t)}{h }.
\end{aligned}
\end{equation}
Inserting \eqref{MFE-trigonometric} into this result and  comparing the
coefficients of $\mathrm{e}^{\mathrm{i}(\boldsymbol{k}\cdot\boldsymbol{\omega}) t}$ implies
\begin{equation}\label{MFE-zetaeta}%
\begin{aligned}&\boldsymbol{\eta}^{\boldsymbol{k}}(\epsilon t)=\widehat{L}^{\boldsymbol{k}}(\epsilon
h  D) \boldsymbol{\zeta}^{\boldsymbol{k}}(\epsilon t),
\end{aligned} %
\end{equation}
where  $\widehat{L}^{\boldsymbol{k}}(\epsilon
h  D)$ is defined by \begin{equation}\label{new L1}
 \begin{array}[c]{ll}
\widehat{L}^{\boldsymbol{k}}(\epsilon
h  D):=& \Big((I+h \Omega\cot(h \Omega)) \big(\mathrm{e}^{\mathrm{i}(\boldsymbol{k} \cdot\boldsymbol{\omega})
h }\mathrm{e}^{\epsilon
h  D}+\mathrm{e}^{-\mathrm{i}(\boldsymbol{k} \cdot\boldsymbol{\omega})
h }\mathrm{e}^{-\epsilon
h  D}\big)\\&-2(\cos(h \Omega)+h \Omega\csc(h \Omega))\Big) \big( h \big(\mathrm{e}^{\mathrm{i}(\boldsymbol{k} \cdot\boldsymbol{\omega})
h }\mathrm{e}^{\epsilon
h  D}-\mathrm{e}^{-\mathrm{i}(\boldsymbol{k} \cdot\boldsymbol{\omega})
h }\mathrm{e}^{-\epsilon
h  D}\big)\big)^{-1}\end{array}\end{equation}
with the
differential operator  $D$ (see \cite{hairer2006}). Therefore, we only need to show the existence and uniqueness of $\boldsymbol{\zeta}^{\boldsymbol{k}}$ and then $\boldsymbol{\eta}^{\boldsymbol{k}}$ are obtained by \eqref{MFE-zetaeta}.

According to the first   term  of     \eqref{methodF}, we consider
 \begin{equation}\label{methods}
\begin{aligned}
\boldsymbol{\tilde{q}}(t+h )-2\cos(h \Omega)\boldsymbol{\tilde{q}}(t)+\boldsymbol{\tilde{q}}(t-h )=h ^2\sinc(  h \Omega) \tilde{f}\big( \boldsymbol{\tilde{q}}(t)\big) \quad \textmd{up to a small defect}.
\end{aligned}
\end{equation}
In the analysis, we  assume that
$f'(0)=0$ since the linearization of $f(\boldsymbol{\tilde{q}})$ can be moved to the linear part. This procedure gives a new nonlinearity $g(\boldsymbol{\tilde{q}})=f(\boldsymbol{\tilde{q}})-f'(0)\boldsymbol{\tilde{q}}$ and it satisfies $g'(0)=0$.
 Inserting \eqref{MFE-trigonometric}  into   \eqref{methods}, expanding the right-hand side into a Taylor
series around zero and comparing the coefficients of
$\mathrm{e}^{\mathrm{i}(\boldsymbol{k}\cdot\boldsymbol{\omega}) t}$ leads to
\begin{equation}\label{ljkqp}
\begin{aligned}
L^{\boldsymbol{k}}(\epsilon
h  D) \zeta_j^{\boldsymbol{k}}(\epsilon t)=&h ^2\sinc(  h \Omega)\sum\limits_{m\geq
2}\frac{f^{(m)}(0)}{m!}
\sum\limits_{\boldsymbol{k}^1+\cdots+\boldsymbol{k}^m=\boldsymbol{k}}\sum\limits_{j_1+\cdots+j_m\equiv j
\textmd{mod}2M}'  \big(\zeta_{j_1}^{\boldsymbol{k}^1}\cdot\ldots\cdot
\zeta_{j_m}^{\boldsymbol{k}^m}\big) (\epsilon t),\\
\end{aligned}
\end{equation}
where    the right-hand side is obtained as in \cite{Cohen08-0} and $L^{\boldsymbol{k}}(\epsilon
h  D)$ is defined by \begin{equation}\label{new L2}
 \begin{aligned}
L_j^{\boldsymbol{k}}(\epsilon
h  D) \zeta_j^{\boldsymbol{k}}(\epsilon t):=&  \mathrm{e}^{\mathrm{i}(\boldsymbol{k} \cdot\boldsymbol{\omega})
h }\mathrm{e}^{\epsilon
h  D}\zeta_j^{\boldsymbol{k}}(\epsilon t)-2\cos(h \omega_j)\zeta_j^{\boldsymbol{k}}(\epsilon t)+\mathrm{e}^{-\mathrm{i}(\boldsymbol{k} \cdot\boldsymbol{\omega})
h }\mathrm{e}^{-\epsilon
h  D}\zeta_j^{\boldsymbol{k}}(\epsilon t)\\
=&  4   s_{\langle j\rangle+\boldsymbol{k}}s_{\langle
j\rangle-\boldsymbol{k}}\zeta_j^{\boldsymbol{k}}(\epsilon t)+2\textmd{i} h  \epsilon
  s_{2k} \dot{\zeta}_j^{\boldsymbol{k}}(\epsilon t)+  h ^2 \epsilon^2
  c_{2k} \ddot{\zeta}_j^{\boldsymbol{k}}(\epsilon t)+\cdots\end{aligned}\end{equation}
  with $s_{\boldsymbol{k}} :=\sin(
\frac{h }{2}(\boldsymbol{k}\cdot\boldsymbol{\omega}))$ and $ c_{\boldsymbol{k}} :=\cos(
\frac{h }{2}(\boldsymbol{k}\cdot\boldsymbol{\omega}))$.
It is noted  that  here and after all the encountered derivatives of  $\boldsymbol{\zeta}^{\boldsymbol{k}}$ are w.r.t the slow time  $\tilde{t}$ and  $\tilde{f}^{(m)}(\boldsymbol{q})$ indicates the $m$-th derivative  of $\tilde{f} (\boldsymbol{q})$ w.r.t. $\boldsymbol{q}$.
For different $ \boldsymbol{k}$, it can be checked that
\begin{equation}\label{new L2r}
L_j^{\boldsymbol{k}}(\epsilon
h  D) \zeta_j^{\boldsymbol{k}}(\epsilon t)=\left\{\begin{aligned} & \pm2 \textmd{i} \epsilon  h  \sin(h   \omega_j)
\dot{\zeta}_j^{\pm \langle j\rangle}+  h ^2 \epsilon^2
\cos(h   \omega_j) \ddot{\zeta}_j^{\pm \langle j\rangle}(\epsilon t)+\cdots,\ \boldsymbol{k}=\pm \langle j\rangle;\\
&4   s_{\langle j\rangle+\boldsymbol{k}}s_{\langle
j\rangle-\boldsymbol{k}}\zeta_j^{\boldsymbol{k}}(\epsilon t)+2\textmd{i} h  \epsilon
  s_{2\boldsymbol{k}} \dot{\zeta}_j^{\boldsymbol{k}}(\epsilon t)+\cdots,     \quad  (j,\boldsymbol{k})\in \mathcal{S}_{\epsilon,h };
  \\
&0,  \qquad \qquad  \qquad    \qquad \qquad  \qquad    \qquad \qquad  \qquad    \ \ (j,\boldsymbol{k})\in \mathcal{R}_{\epsilon,h }.
\end{aligned}\right.
\end{equation}
It is noted that if $ \boldsymbol{k}\in \mathcal{R}_{\epsilon,h }$,
the non-resonance condition  \eqref{non-resonance cond}   ensures
that the defect in simply setting $ \zeta_j^{\boldsymbol{k}}\equiv 0$ is of size
$\mathcal{O}(\epsilon^{N+1})$ in an appropriate Sobolev-type norm.  Therefore, we set $ \zeta_j^{\boldsymbol{k}}\equiv 0$  for this case and so that $L_j^{\boldsymbol{k}}( \epsilon
h  D) \zeta_j^{\boldsymbol{k}}(\epsilon t)=0$.
Based on  \eqref{new L2r} and \eqref{ljkqp},  we get the formal modulation equations
  for the modulated functions  $ \zeta_j^{\boldsymbol{k}}$. We should point out that
  the dominating
terms in the formal modulation equations are  \begin{equation}\label{domi term}\pm2 \textmd{i} \epsilon  h  \sin(h   \omega_j)
\dot{\zeta}_j^{\pm \langle j\rangle} \ \ (\boldsymbol{k}=\pm \langle j\rangle)\ \  \textmd{and}\ \  4   s_{\langle j\rangle+\boldsymbol{k}}s_{\langle
j\rangle-\boldsymbol{k}}\zeta_j^{\boldsymbol{k}}(\epsilon t)\ \   ( (j,\boldsymbol{k})\in \mathcal{S}_{\epsilon,h }).\end{equation}

  Since $
\dot{\zeta}_j^{\pm \langle j\rangle}$ appears in the dominating
term for $\boldsymbol{k}=\pm \langle j\rangle$, we need to derive the initial values
 for $\zeta_j^{\pm \langle j\rangle}(h )$ which are deduced by requiring $\boldsymbol{\tilde{q}} (h ) =
\boldsymbol{q}_1$ and $\boldsymbol{\tilde{p}} (h ) = \boldsymbol{p}_1$. From the first requirement, it follows that
\begin{equation}\label{initial pl}
\begin{aligned}&\zeta_j^{\langle
j\rangle}(h )+\zeta_j^{-\langle j\rangle}(h )=\mathrm{e}^{-\mathrm{i}\omega_j
h }(\boldsymbol{q}_1)_j-\mathrm{e}^{-\mathrm{i}\omega_j
h }\sum\limits_{\boldsymbol{k} \neq
\pm  \langle j\rangle} \mathrm{e}^{\mathrm{i}(\boldsymbol{k} \cdot\boldsymbol{\omega})
h }\zeta_j^{\boldsymbol{k}}(h ).
\end{aligned}
\end{equation}
Moreover, according to the formula \eqref{MFE-zetaeta} and by using the Taylor series of $\widehat{L}^{\pm\langle j\rangle}$ \eqref{new L1}, we have
 \begin{equation}\label{impor-zetaeta}%
\begin{aligned}&\eta_j^{\pm\langle j\rangle}=\widehat{L}_j^{\pm\langle j\rangle}(\epsilon
h  D) \zeta_j^{\pm\langle j\rangle}=\pm\textmd{i} \omega_j\zeta_j^{\pm\langle
j\rangle}+  \epsilon \dot{\zeta}_j^{\pm\langle
j\rangle}\pm\cdots.
\end{aligned} %
\end{equation}
Then the requirement $\boldsymbol{\tilde{p}} (h ) = \boldsymbol{p}_1$ yields
$$ \begin{aligned}\eta_j^{\langle
j\rangle}(h )+\eta_j^{-\langle j\rangle}(h )=&\mathrm{e}^{-\mathrm{i}\omega_j
h }(\boldsymbol{p}_1)_j-\mathrm{e}^{-\mathrm{i}\omega_j
h }\sum\limits_{\boldsymbol{k} \neq
\pm  \langle j\rangle}\mathrm{e}^{\mathrm{i}(\boldsymbol{k} \cdot\boldsymbol{\omega})
h } \eta_j^{\boldsymbol{k}}(h )\\
= &\textmd{i} \omega_j(\zeta_j^{\langle
j\rangle}(h )-\zeta_j^{-\langle
j\rangle}(h ))+2  \epsilon
 \dot{\zeta}_j^{\pm\langle
j\rangle}(h )+\cdots.\end{aligned} $$
Thus we find
\begin{equation}\label{initial mi}
\begin{aligned}&\textmd{i} \omega_j(\zeta_j^{\langle
j\rangle}(h )-\zeta_j^{-\langle
j\rangle}(h ))= \mathrm{e}^{-\mathrm{i}\omega_j
h }(\boldsymbol{p}_1)_j-\mathrm{e}^{-\mathrm{i}\omega_j
h }\sum\limits_{\boldsymbol{k} \neq
\pm  \langle j\rangle} \mathrm{e}^{\mathrm{i}(\boldsymbol{k} \cdot\boldsymbol{\omega})
h }\eta_j^{\boldsymbol{k}}(h ) - 2  \epsilon
 \dot{\zeta}_j^{\pm\langle
j\rangle}(h )+\cdots.\end{aligned}
\end{equation}
This formula and  \eqref{initial pl}  determine the
initial values $\zeta_j^{\pm \langle j\rangle}(h )$.

 \textbf{Step 2. Reverse Picard iteration and bounds of the coefficient functions.}

For the specific construction of the coefficient
functions $\boldsymbol{\zeta}^{\boldsymbol{k}}$, we use  an iterative procedure and show that
  after $4N$ iteration steps, the
defects in \eqref{ljkqp}, \eqref{initial pl} and \eqref{initial mi}
are of magnitude $\mathcal{O}(\epsilon^{N+1})$ in the $H^{s}$ norm. The main idea in
the iteration  is to keep the dominant term \eqref{domi term}  on the
left-hand side and we call the procedure as reverse Picard iteration (\cite{Cohen08,Cohen08-1}).

$\bullet$ \textbf{Rescaling and new notations.} Noticing that the bound \eqref{bound modu func} is given for
 $\frac{\boldsymbol{\omega}^{|\boldsymbol{k}| }}{\epsilon^{[[\boldsymbol{k}]]}} \boldsymbol{\zeta}^{\boldsymbol{k}}(\epsilon
t) $,
we consider it as a
rescaling of $\boldsymbol{\zeta}^{\boldsymbol{k}}$:
$
 \tilde{\zeta}_{j}^{\boldsymbol{k}}:=\frac{\boldsymbol{\omega}^{|\boldsymbol{k}| }}{\epsilon^{[[\boldsymbol{k}]]}}\zeta_j^{\boldsymbol{k}},\
\tilde{\boldsymbol{\zeta}}^{\boldsymbol{k}}=\big(\tilde{\zeta}_{j}^{\boldsymbol{k}}\big)_{|j|\leq
M}=\frac{\boldsymbol{\omega}^{|\boldsymbol{k}| }}{\epsilon^{[[\boldsymbol{k}]]}}\boldsymbol{\zeta}^{\boldsymbol{k}}$
 in the space $H^{s} = (H^{s})^{\mathcal{K}} = \{\tilde{\boldsymbol{\zeta}}=
(\tilde{\boldsymbol{\zeta}}^{\boldsymbol{k}})_{\boldsymbol{k}\in \mathcal{K}}: \tilde{\boldsymbol{\zeta}}^{\boldsymbol{k}}\in H^{s}\}$ as well as its norm $|||\tilde{\boldsymbol{\zeta}}|||_{H^{s}}^2=\sum\limits_{\boldsymbol{k}\in
\mathcal{K}} \norm{\boldsymbol{\tilde{\zeta}}^{\boldsymbol{k}}}_{H^{s}}^2$, where $\mathcal{K}:=\{\boldsymbol{k}=(k_l)_{\abs{l}\leq M}\ \textmd{with integers} \
k_l:\ \norm{\boldsymbol{k}}\leq K\}$ with $K = 2N.$
In order to express the non-linearity of \eqref{ljkqp} in these
rescaled variables, we define a rescaling nonlinear function by
$\boldsymbol{g}^{\boldsymbol{k}}=(g_j^{\boldsymbol{k}})$ with
\begin{equation}\label{gfin r}
\begin{aligned}g_j^{\boldsymbol{k}}\big(\tilde{\boldsymbol{\zeta}}(\tilde{t})\big) :=&\frac{\boldsymbol{\omega}^{|\boldsymbol{k}| }}{\epsilon^{[[\boldsymbol{k}]]}}\sum\limits_{m= 2}^N\frac{f^{(m)}(0)}{m!}
\sum\limits_{\boldsymbol{k}^1+\cdots+\boldsymbol{k}^m=\boldsymbol{k}}\frac{\epsilon^{[[\boldsymbol{k}^1]]+\cdots+[[\boldsymbol{k}^m]]}}{\boldsymbol{\omega}^{|\boldsymbol{k}^1|+\cdots+|\boldsymbol{k}^m|}}
\sum\limits_{j_1+\cdots+j_m\equiv j \textmd{mod} 2M}'
\big(\tilde{\zeta}_{j_1}^{\boldsymbol{k}^1}\cdot\ldots\cdot
\tilde{\zeta}_{j_m}^{\boldsymbol{k}^m}\big)(\tilde{t}).
\end{aligned}
\end{equation}


It is remarked that all the $\boldsymbol{k}\in \mathcal{K}$ can be divided into three kinds: (I) $\boldsymbol{k}=\pm \langle
j\rangle$; (II)
$\boldsymbol{k} \neq\pm\langle j\rangle $ but \eqref{inequa} holding ($(j,\boldsymbol{k})\in \mathcal{S}_{\epsilon,h }$); (III) $\boldsymbol{k} \neq\pm\langle j\rangle $ but \eqref{inequa} not  holding  ($(j,\boldsymbol{k})\in \mathcal{R}_{\epsilon,h }$). Since for the case (III), it is required that $\tilde{\zeta}_j^{\boldsymbol{k}}=0$,   we split $\tilde{\boldsymbol{\zeta}}^{\boldsymbol{k}}$ into two parts: $ \tilde{\boldsymbol{\zeta}}^{\boldsymbol{k}}=\boldsymbol{a}^{\boldsymbol{k}} + \boldsymbol{b}^{\boldsymbol{k}}$, where
$a_j^{\boldsymbol{k}}=\tilde{\zeta}_j^{\boldsymbol{k}}$ if $\boldsymbol{k}=\pm \langle
j\rangle$,  and $b_j^{\boldsymbol{k}}=\tilde{\zeta}_j^{\boldsymbol{k}}$ if $(j,\boldsymbol{k})\in \mathcal{S}_{\epsilon,h }$.
It is remarked that
$|||\tilde{\boldsymbol{\zeta}}^{\boldsymbol{k}}|||_{H^{s}}^2=|||\boldsymbol{a}^{\boldsymbol{k}}|||_{H^{s}}^2+|||\boldsymbol{b}^{\boldsymbol{k}}|||_{H^{s}}^2$.
 To make the presentation be compact and concise, we introduce two differential operators $A, B$ to  express \eqref{new L2r} in the new rescaling for the cases (I) and (II) respectively:
\begin{equation}\label{AB}
\begin{aligned}& (A\boldsymbol{a})_j^{\pm \langle j\rangle}:=\frac{1}{\pm2 \textmd{i} \epsilon  h  \sin(h   \omega_j)}\big(h ^2 \epsilon^2
\cos(h   \omega_j)
 \ddot{a}_j^{\pm  \langle
j\rangle}+\cdots\big),  \\
&(B\boldsymbol{b})_j^{\boldsymbol{k}}:=\frac{1}{4   s_{\langle j\rangle+\boldsymbol{k}}s_{\langle
j\rangle-\boldsymbol{k}}}\big( 2\textmd{i} h  \epsilon
  s_{2\boldsymbol{k}}  \dot{b}_j^{\boldsymbol{k}} +h ^2 \epsilon^2
  c_{2\boldsymbol{k}} \ddot{b}_j^{\boldsymbol{k}}+\cdots\big)
\ \  \textmd{for} \  (j,\boldsymbol{k})\in \mathcal{S}_{\epsilon,h },
\end{aligned}
\end{equation}
where  the dots in the above two formulae represent a truncation after the $\epsilon^N$ terms.
Combining the non-linearity in \eqref{ljkqp} with the    function  $\boldsymbol{g}$ \eqref{gfin r}, we define two new functions
$\textbf{F}$ and $\textbf{G}$
\begin{equation}\label{FG}
\begin{aligned} &F_j^{\pm \langle
j\rangle}(\boldsymbol{a},\boldsymbol{b}):=\frac{1}{\pm \textmd{i}\epsilon}g_j^{\pm
\langle j\rangle}\big(\tilde{\boldsymbol{\zeta}}\big),\
  G_j^{\boldsymbol{k}}(\boldsymbol{a},\boldsymbol{b}):=-\frac{h ^2(\omega_j+|\boldsymbol{k}\cdot \boldsymbol{\omega}|)}{4s_{\langle j\rangle+\boldsymbol{k}}s_{\langle
j\rangle-\boldsymbol{k}}}g_j^{\boldsymbol{k}}\big(\tilde{\boldsymbol{\zeta}}\big)\   \textmd{for} \  (j,\boldsymbol{k})\in \mathcal{S}_{\epsilon,h }.
\end{aligned}
\end{equation}
Furthermore, let $(\Lambda \zeta)_j^{\boldsymbol{k}}:=(\omega_j+|\boldsymbol{k}\cdot \boldsymbol{\omega}|) \zeta_j^{\boldsymbol{k}} $ and $   (\Upsilon \zeta)_j^{\boldsymbol{k}}:=\sinc (h \omega_j) \zeta_j^{\boldsymbol{k}}.$

$\bullet$ \textbf{Reverse Picard iteration.} With the above  notations, we are in a position to present the  reverse Picard iterations in the new rescaling  by keeping the dominant term  on the
left-hand side. The iterations can be formulated in a very compact form
\begin{equation}\label{Abstruct Pic ite}
\begin{aligned} & \dot{\boldsymbol{a}}^{(n+1)}(\tilde{t})=\Lambda^{-1}\mathbf{F}(\boldsymbol{a}^{(n)}(\tilde{t}),\boldsymbol{b}^{(n)}(\tilde{t}))
-A\boldsymbol{a}^{(n)}(\tilde{t}), \ \
 \boldsymbol{b}^{(n+1)}(\tilde{t})=\Lambda^{-1}\Upsilon \mathbf{G}(\boldsymbol{a}^{(n)}(\tilde{t}),\boldsymbol{b}^{(n)}(\tilde{t}))-B\boldsymbol{b}^{(n)}(\tilde{t}),
\end{aligned}
\end{equation}
where  we denote the $n$th iteration  by  $[\cdot]^{(n)}$. It can be seen that
an initial value problem of
first-order differential equations for $\boldsymbol{a}^{(n+1)}(\tilde{t})$  and algebraic equations for
 $\boldsymbol{b}^{(n+1)}(\tilde{t})$   are contained
in each iteration step.
The starting iterates are chosen as $\boldsymbol{b}^{(0)}(\tilde{t})=0$
and $\boldsymbol{a}^{(0)}(\tilde{t})=\boldsymbol{a}^{(0)}(h)$  determined below.

On the other hand,  the reverse Picard iterations of initial values  \eqref{initial pl} and \eqref{initial mi}
  can also be written in a compact form
\begin{equation}\label{abs ini con}
\begin{aligned} &\boldsymbol{a}^{(n+1)}(h )=\boldsymbol{\nu}+P\boldsymbol{b}^{(n)}(h )+Q\boldsymbol{a}^{(n)}(h ),
\end{aligned}
\end{equation}
where $\boldsymbol{\nu}$ is defined by:
$\nu_j^{\pm\langle j\rangle
}=\frac{\omega_j}{\epsilon}\mathrm{e}^{-\mathrm{i}\omega_j
h }\Big(\frac{1}{2}(\boldsymbol{q}_1)_j \mp\frac{\textmd{i}}{2\omega_j}(\boldsymbol{p}_1)_j\Big)
$
 and by  the convergence result \eqref{F bound10}, it is immediately arrived  that   $\boldsymbol{\nu}$ is bounded in
$H^{s}$. The operators $P$ and $Q$   are defined by
\begin{equation}\label{PQ}
\begin{aligned}
(P\boldsymbol{b})_j^{\pm\langle j\rangle}(h ):=&-\frac{\omega_j}{2\epsilon
}\mathrm{e}^{-\mathrm{i}\omega_j
h }\sum\limits_{\boldsymbol{k} \neq \pm\langle j\rangle}\big(1\pm \frac{1}{\textmd{i}  \omega_j}\widehat{L}^{\boldsymbol{k}}\big)\frac{\epsilon^{[[\boldsymbol{k}]]}}{\boldsymbol{\omega}^{|\boldsymbol{k}| }} \mathrm{e}^{\mathrm{i}(\boldsymbol{k} \cdot\boldsymbol{\omega})
h }b_j^{\boldsymbol{k}}(h ),\
 (Q\boldsymbol{a})_j^{\pm\langle
j\rangle}(h ):=\frac{\pm\epsilon}{\textmd{i}   \omega_j}
\Big(  \dot{a}_j^{\pm\langle
j\rangle}(h )+\cdots\Big).
\end{aligned}
\end{equation}
For the starting values of \eqref{abs ini con}, we choose
  $\boldsymbol{a}^{(0)}(h ) = \boldsymbol{\nu}$ and $\boldsymbol{b}^{(0)}(h ) =
0$. 
In these reverse Picard iterations, we terminate the procedure when $n=4N.$

 $\bullet$ \textbf{Proof of \eqref{MFE-trigonometric}.}
So far, we have presented the reverse Picard iterations \eqref{Abstruct Pic ite}-\eqref{abs ini con} of the modulation functions and initial values. These and \eqref{MFE-zetaeta} show the  existence and uniqueness of \eqref{MFE-trigonometric}. In the rest parts of this step, we are going to show the bounds \eqref{bound TMFE}-\eqref{bound modu func} of the   expansions and modulation functions. To this end, we need to
deduce the estimates of   the functions  and operators used in
\eqref{Abstruct Pic ite}-\eqref{abs ini con}.

 $\bullet$ \textbf{Estimates of $\mathbf{F},\mathbf{G}$ \eqref{FG}.}   From the results given in Section 3.5 of  \cite{Cohen08},
it follows that
\begin{equation*}
\begin{aligned} &
\sum\limits_{|j|\leq M} \norm{\boldsymbol{g}^{\pm\langle j\rangle
}(\tilde{\boldsymbol{\zeta}})}_{H^{s}}^2\leq \epsilon^3 P_1(|||\tilde{\boldsymbol{\zeta}}|||_{H^{s}}^2),\ \ \sum\limits_{\boldsymbol{k}\in \mathcal{K}}
\norm{\boldsymbol{g}^{\boldsymbol{k}}(\tilde{\boldsymbol{\zeta}})}_{H^{s}}^2\leq \epsilon P_2(|||\tilde{\boldsymbol{\zeta}}|||_{H^{s}}^2),
\end{aligned}
\end{equation*}
where $P_1$ and $P_2$ are both polynomials with coefficients bounded
independently of $\epsilon, h $ and $M$.
Based on these inequalities and \eqref{inequa}, it is easy to get the bounds of $\mathbf{F},\mathbf{G}$ defined in \eqref{FG}:
$|||\mathbf{F}|||_{H^{s}}\leq C\epsilon^{1/2}, |||\mathbf{G}|||_{H^{s}}\leq C.$
 Moreover,    we deduce that
$$ |||\Upsilon^{-1} \Lambda^{-1}\mathbf{F}|||_{H^{s}}\leq|||\frac{1}{\pm \textmd{i}\epsilon}
\frac{h }{\sin (h \omega_j)}\frac{\omega_j}{2\omega_j} \boldsymbol{g}^{\pm\langle j\rangle
}|||_{H^{s}}\leq C.$$

$\bullet$  \textbf{Estimates of $A,B$ \eqref{AB}.} In view of   \eqref{inequa}, it is immediately arrived that
\begin{equation}\label{AB bound}
\begin{aligned} &|||A\boldsymbol{a}|||_{H^{s}}\leq C
\sum\limits_{l=2}^N h ^{l-2}\epsilon^{l-3/2}|||\boldsymbol{a}^{(l)}|||_{H^{s}},\ \
 |||B\boldsymbol{b}|||_{H^{s}}\leq C \epsilon^{1/2}|||   \dot{\boldsymbol{b}} |||_{H^{s}}+C
\sum\limits_{l=2}^Nh ^{l-2}\epsilon^{l-1/2}||| \boldsymbol{b}^{(l)}|||_{H^{s}}.
\end{aligned}
\end{equation}

$\bullet$  \textbf{Estimates of $P,Q$ \eqref{PQ}.}
For the bound of $P$, we have
\begin{equation*}
\begin{aligned}
&|||(P\boldsymbol{b})(h )|||^2_{H^{s}}=\sum\limits_{\norm{\langle j\rangle} \leq K}
\sum\limits_{|j|\leq M}''\omega_j^{2s}
\abs{\frac{\omega_j}{2\epsilon }\sum\limits_{\boldsymbol{k} \neq \pm\langle
j\rangle}\big(1\pm \frac{1}{\textmd{i} h
\omega_j}\widehat{L}_j^{\boldsymbol{k}}\big)\frac{\epsilon^{[[\boldsymbol{k}]]}}{\boldsymbol{\omega}^{|\boldsymbol{k}| }}\mathrm{e}^{\mathrm{i}(\boldsymbol{k} \cdot\boldsymbol{\omega})
h }\boldsymbol{b}_j^{\boldsymbol{k}}(0)}^2\\
\leq&\frac{C}{4\epsilon^2 } \sum\limits_{|j|\leq
M}''  \omega_j^{2s+2} \sum\limits_{\boldsymbol{k} \neq \pm\langle
j\rangle}\abs{\frac{\sin(h  \omega_j)}{h  }\frac{\epsilon^{[[\boldsymbol{k}]]}}{\boldsymbol{\omega}^{|\boldsymbol{k}| }}\frac{\boldsymbol{b}_j^{\boldsymbol{k}}(0)}{\sinc(h  \omega_j)}}^2
\leq \frac{C}{4\epsilon^2 } \sum\limits_{|j|\leq
M}''\sum\limits_{\boldsymbol{k} \neq \pm\langle j\rangle}
\frac{\epsilon^{2[[\boldsymbol{k}]]}}{\boldsymbol{\omega}^{2|\boldsymbol{k}| }}  \sum\limits_{\boldsymbol{k} \neq
\pm\langle j\rangle}\omega_j^{2s+2}\abs{\frac{ \boldsymbol{b}_j^{\boldsymbol{k}}(0)}{\sinc(h  \omega_j)}}^2\\
\leq&C \sum\limits_{\boldsymbol{k} \neq \pm\langle j\rangle}
\frac{\epsilon^{2[[\boldsymbol{k}]]}}{4\epsilon^2 }  \boldsymbol{\omega}^{-2|\boldsymbol{k}| }
|||(\Lambda \boldsymbol{b})(0)|||^2_{H^{1+s}}
\leq C  |||(\Upsilon^{-1}\Lambda \boldsymbol{b})(h)|||^2_{1+H^{s}},
\end{aligned}
\end{equation*}
where we used  the Cauchy-Schwarz inequality as well as  the result
$\sum\limits_{\norm{\boldsymbol{k}}\leq
K}\boldsymbol{\omega}^{-2|\boldsymbol{k}| }\leq C$ (Lemma 2 of \cite{Cohen08}). The bound of $Q$ is obvious as
$|||(Q\boldsymbol{a})(h )|||_{H^{s}}\leq
 C
\sum\limits_{l=1}^N h ^{l-1}\epsilon^{l}|||\boldsymbol{a}^{(l)}|||_{H^{s}}.
$

$\bullet$  \textbf{Proof of \eqref{bound modu func} and \eqref{bound TMFE}.}
By  the above estimates, the  condition \eqref{inequa}
and the  iterative  procedure, we can prove   by induction that
the iterates $\boldsymbol{a}^{(n)},\ \boldsymbol{b}^{(n)}$  and their derivatives
with respect to the slow time $\tilde{t}$   satisfy
\begin{equation}\label{Bounds func zeta}
\begin{array}{ll}   &|||\boldsymbol{a}^{(n)}(h )|||_{H^{s}}\leq C,\   |||\Lambda  \dot{\boldsymbol{a}}^{(n)}(\epsilon t)|||_{H^{s}}\leq
C\epsilon^{1/2},\ \
  |||\Upsilon^{-1} \dot{\boldsymbol{a}}^{(n)}(\epsilon t )|||_{H^{s}}\leq C,\
 |||\Upsilon^{-1}\Lambda \boldsymbol{b}^{(n)} (\epsilon t )|||_{H^{s}}\leq C,
\end{array}
\end{equation}
where $n=1,2,\ldots,4N$ and $C$ is independent of $\epsilon, h, M$ but dependent on $N$.
In what follows, we use the abbreviations $\boldsymbol{a}$ and $\boldsymbol{b}$ to denote the $(4N)$-th iterates.
 These  estimates  immediately lead to  \begin{equation}\label{Bound an}\begin{aligned}&|||\boldsymbol{a}(\epsilon t)|||_{H^{s}}\leq C,\ \  |||\boldsymbol{b}  (\epsilon t )|||_{H^{1+s}}
\lesssim|||\Lambda\boldsymbol{b}  (\epsilon t )|||_{H^{s}}
  =  |||\Upsilon \Upsilon^{-1}\Lambda\boldsymbol{b}  (\epsilon t )|||_{H^{s}}
\lesssim ||| \Upsilon^{-1}\Lambda\boldsymbol{b}  (\epsilon t )|||_{H^{s}}\lesssim C,
\end{aligned}\end{equation}
which shows
the statement \eqref{bound modu func} and the result   \begin{equation}\label{Bounds am0}
\begin{array}{ll}   &|||  \boldsymbol{a}  (\epsilon t )-\boldsymbol{a}(h )+\boldsymbol{b}  (\epsilon t )|||_{H^{1+s}}
\lesssim
|||  \boldsymbol{\dot{a}}  ( \epsilon t_s )|||_{H^{1+s}}+|||\boldsymbol{b}  (\epsilon t )|||_{H^{1+s}}
\lesssim
  |||\Lambda \boldsymbol{\dot{a}}  ( \epsilon t_s )|||_{H^{s}}+C
\lesssim C
\end{array}
\end{equation}
with some intermediate time value $t_s $ which may vary line by line.
Based on \eqref{Bounds am0} and \eqref{Bounds func zeta},
the estimate \eqref{bound TMFE} of $\boldsymbol{\tilde{q}}(t)$ is deduced as follows
\begin{equation*}
\begin{aligned}
\norm{\boldsymbol{\tilde{q}}(t)}^2_{H^{1+s}}
\lesssim&  \sum\limits_{|j|\leq
M}''\omega_j^{2+2s}\Big(\frac{\epsilon }{\omega_j}\big(\abs{a_j^{\langle j\rangle}(h )}+\abs{a_j^{-\langle j\rangle}(h )}\big)
+   \sum\limits_{\norm{\boldsymbol{k}}\leq2N} \frac{\epsilon^{[[\boldsymbol{k}]]}}{\boldsymbol{\omega}^{|\boldsymbol{k}| }} \abs{a_j^{\boldsymbol{k}}(\epsilon t)+b_j^{\boldsymbol{k}}(\epsilon t)-a_j^{\boldsymbol{k}}(h )}\Big)^2 \\
\lesssim&  4\epsilon^2 |||\boldsymbol{a}(h )|||_{H^{s}}+ \epsilon^2 \sum\limits_{|j|\leq
M}''\omega_j^{2+2s} \sum\limits_{\norm{\boldsymbol{k}}\leq2N}   \abs{a_j^{\boldsymbol{k}}(\epsilon t)+b_j^{\boldsymbol{k}}(\epsilon t)-a_j^{\boldsymbol{k}}(h )}^2\\
\lesssim&  4\epsilon^2 |||\boldsymbol{a}(h )|||_{H^{s}}+ \epsilon^2 |||  \boldsymbol{a}  (\epsilon t )-\boldsymbol{a}(h )+\boldsymbol{b}  (\epsilon t )|||_{H^{1+s}}\lesssim \epsilon^2.
\end{aligned}
\end{equation*}
The result \eqref{bound TMFE} of $\boldsymbol{\tilde{p}}(t)$ is estimated in the $H^{s}$-norm and thus it is easy to
prove it by using the connection \eqref{MFE-zetaeta}, the bound  \eqref{Bound an}  and similar arguments presented above.

$\bullet$  \textbf{Proof of \eqref{jth TMFE}.}
To prove the statement  \eqref{jth TMFE},
we split it into three parts
 \begin{equation*}
\begin{aligned}
r_j=&\underbrace{\sum\limits_{\norm{\boldsymbol{k}}=1} \mathrm{e}^{\mathrm{i}(\boldsymbol{k}\cdot\boldsymbol{\omega}) t}\zeta_j^{\boldsymbol{k}}-\zeta_j^{\langle j\rangle}
\mathrm{e}^{\mathrm{i} \omega_j t}-\zeta_j^{-\langle
j\rangle} \mathrm{e}^{-\mathrm{i} \omega_j t}}_{=:\textmd{Part I}^{r}}
+\underbrace{\sum\limits_{\norm{\boldsymbol{k}}=2} \mathrm{e}^{\mathrm{i}(\boldsymbol{k}\cdot\boldsymbol{\omega}) t}\zeta_j^{\boldsymbol{k}}}_{=:\textmd{Part II}^{r}}+\underbrace{\sum\limits_{3\leq \norm{\boldsymbol{k}}\leq 2N} \mathrm{e}^{\mathrm{i}(\boldsymbol{k}\cdot\boldsymbol{\omega}) t}\zeta_j^{\boldsymbol{k}}}_{=:\textmd{Part III}^{r}}.
\end{aligned}
\end{equation*}
For the Part III$^{r}$,  the bound \eqref{bound modu func} gives that  $\norm{\textmd{Part III}^{r}}_{H^{1+s}}\leq C\epsilon^2.$
Considering
$$(\Upsilon^{-1}\Lambda \boldsymbol{b}^{(n+1)})^{\boldsymbol{k}}=\mathbf{G}^{\boldsymbol{k}}(\boldsymbol{a}^{(n)},\boldsymbol{b}^{(n)})-\big(\Upsilon^{-1}\Lambda B\boldsymbol{b}^{(n)}(\tilde{t})\big)^{\boldsymbol{k}}  \ \ \textmd{with}\ \ \boldsymbol{b}^{(0 )}=0,$$
and the fact that $$\sum\limits_{\norm{\boldsymbol{k}}=1}\norm{G_l^{\boldsymbol{k}}}_{H^{s}}=\sum\limits_{ \boldsymbol{k}=\langle j\rangle}\norm{-\frac{h ^2(\omega_l+|\boldsymbol{k}\cdot \boldsymbol{\omega}|)}{4s_{\langle l\rangle+\boldsymbol{k}}s_{\langle
l\rangle-\boldsymbol{k}}}\tilde{g}_l^{\boldsymbol{k}}\big(\tilde{\boldsymbol{\zeta}}\big)}_{H^{s}} \lesssim
\sum\limits_{ |j|\leq
M }\norm{ \frac{1}{\sqrt{\epsilon}}\tilde{g}_l^{\langle j\rangle}\big(\tilde{\boldsymbol{\zeta}}\big)}_{H^{s}}
 \lesssim \epsilon,$$
we get
$\sum\limits_{\norm{\boldsymbol{k}}=1}\norm{(\Upsilon^{-1}\Lambda \boldsymbol{b})^{\boldsymbol{k}}}_{H^{s}}^2 \leq C \epsilon^2
$ and this yields $\norm{\textmd{Part I}^{r}}_{H^{1+s}}\leq C\epsilon^2.$
Based on   the assumption \eqref{another-non-res cond} and the  iterative  procedure, we can find
$\sum\limits_{|j|\leq
M}\sum\limits_{j_1+j_2=j}\sum\limits_{\boldsymbol{k}=\pm\langle
j_1\rangle\pm\langle j_2\rangle}
 \omega_j^{2}|b_j^{\boldsymbol{k}}|^2 \leq C \epsilon$
 and by which,  $\norm{\textmd{Part II}^{r}}_{H^{1+s}}\leq C\epsilon^2.$ The proof of
 \eqref{jth TMFE} is  complete.

 \textbf{Step 3. Defect and remainder.}

$\bullet$  \textbf{Defect in $\boldsymbol{\tilde{q}}$.} We define the defect in \eqref{methods} $$\boldsymbol{\delta}(t):=1/h ^2\sinc^{-1}(  h \Omega) \big(\boldsymbol{\tilde{q}}(t+h )-2\cos(h \Omega)\boldsymbol{\tilde{q}}(t)+\boldsymbol{\tilde{q}}(t-h ) \big)- \tilde{f}\big( \boldsymbol{\tilde{q}}(t)\big).$$
By the construction of the coefficient functions $\boldsymbol{\zeta}^{\boldsymbol{k}}$ of the modulated Fourier expansion $\boldsymbol{\tilde{q}}$, it is seen that the defect comes from three aspects:  the truncation of
the  operator $L^{\boldsymbol{k}}$ after the $\epsilon^N$ term \eqref{AB} (denoted by $\tilde{L}^{\boldsymbol{k}}$), the truncation   of the Taylor expansion of
$f$ after $N$ terms  \eqref{gfin r}, and the finite $4N$ iterations
of \eqref{Abstruct Pic ite} and \eqref{abs ini con} (denoted by $\zeta^{\boldsymbol{k}}_j=(\zeta^{\boldsymbol{k}}_j)^{(4N)}$).
Based on these discussions, we can rewrite the defect in a new form
\begin{equation}\label{dk}
\begin{aligned}
&\boldsymbol{\delta}(t)=\sum\limits_{\norm{\boldsymbol{k}}\leq NK}\boldsymbol{d}^{\boldsymbol{k}}(\epsilon t)
e^{\textmd{i}(\boldsymbol{k}\cdot\boldsymbol{\omega}) t}+R_{N+1}(t),
\end{aligned}
\end{equation}
where
\begin{equation}\label{djk}
\begin{aligned}
d_j^{\boldsymbol{k}}=&\frac{\tilde{L}^{\boldsymbol{k}} \zeta_j^{\boldsymbol{k}}}{h ^2\sinc(  h \omega_j) }-\sum\limits_{m=
2}^N\frac{f^{(m)}(0)}{m!}
\sum\limits_{\boldsymbol{k}^1+\cdots+\boldsymbol{k}^m=\boldsymbol{k}}\sum\limits_{j_1+\cdots+j_m\equiv j
\textmd{mod}2M}'  \big(\zeta_{j_1}^{\boldsymbol{k}^1}\cdot\ldots\cdot
\zeta_{j_m}^{\boldsymbol{k}^m}\big),
\end{aligned}
\end{equation}
and the function
$R_{N+1}$  contains the remainder terms of the Taylor expansion of
$f$ after $N$ terms and of  $L^{\boldsymbol{k}}$ after the $\epsilon^N$ term.  We consider   $\norm{\boldsymbol{k}}\leq NK$, and $\zeta_j^{\boldsymbol{k}}=0$ for
$\norm{\boldsymbol{k}}>K:=2N$.
For the part $R_{N+1}$, by using the bound \eqref{bound TMFE} for the remainder
in the Taylor expansion of $f$ and the estimate \eqref{Bounds func
zeta}  for the $(N + 1)$-th derivative for $\zeta_j^{\boldsymbol{k}}$, it is immediately obtained that
 $\norm{R_{N+1}}_{H^{s}} \leq
C\epsilon^{N+1}.$
For the other part in \eqref{dk},   Section 3.8 of \cite{Cohen08} and Section 6.7 of  \cite{Cohen08-1}
yield that
$\norm{\sum\limits_{\norm{\boldsymbol{k}}\leq NK}\boldsymbol{d}^{\boldsymbol{k}}(\epsilon t)
e^{\textmd{i}(\boldsymbol{k}\cdot\boldsymbol{\omega}) t}}_{H^{s}}^2\leq C
\sum\limits_{\norm{\boldsymbol{k}}\leq NK}\norm{\boldsymbol{\omega}^{|\boldsymbol{k}| }\boldsymbol{d}^{\boldsymbol{k}}(\epsilon
t)}_{H^{s}}^2.
$
In what follows,  we shall show that
\begin{equation}\label{right hand esti}
\begin{aligned}
& \sum\limits_{\norm{\boldsymbol{k}}\leq NK}\norm{\boldsymbol{\omega}^{|\boldsymbol{k}| }\boldsymbol{d}^{\boldsymbol{k}}(\epsilon
t)}_{H^{s}}^2\leq C\epsilon^{2(N+1)}
\end{aligned}
\end{equation}
for three modes: truncated case ($\norm{\boldsymbol{k}}>K:=2N$), near-resonant case ($(j, \boldsymbol{k}) \in\mathcal{R}_{\epsilon,h }$) and non-resonant  case ($(j, \boldsymbol{k}) \in\mathcal{S}_{\epsilon,h }$).
Keeping in mind the fact $\zeta_j^{\boldsymbol{k}}=0$ for the first two  modes, the  defect becomes
\begin{equation*}
\begin{aligned}
d_j^{\boldsymbol{k}}=&\sum\limits_{m= 2}^N\frac{f^{(m)}(0)}{m!}
\sum\limits_{\boldsymbol{k}^1+\cdots+\boldsymbol{k}^m=\boldsymbol{k}}\sum\limits_{j_1+\cdots+j_m\equiv j\
\textmd{mod}\ 2M}' \big(\zeta_{j_1}^{\boldsymbol{k}^1}\cdot\ldots\cdot
\zeta_{j_m}^{\boldsymbol{k}^m}\big),
\end{aligned}
\end{equation*}
and  the same deduction of  \cite{Cohen08-1} shows \eqref{right hand esti}\footnote{The condition  \eqref{non-resonance cond} is used here.}. Thus, the remaining work is to prove
 \eqref{right hand esti} under the condition $(j, \boldsymbol{k}) \in\mathcal{S}_{\epsilon,h }$. To this end, we express the defect \eqref{djk} in the scaled variables
$
\boldsymbol{\omega}^{|\boldsymbol{k}| }d_j^{\boldsymbol{k}}= \epsilon^{[[\boldsymbol{k}]]}
\Big(\frac{1}{h ^2\sinc(  h \omega_j) }\tilde{L}_j^{\boldsymbol{k}}
\tilde{\zeta}_j^{\boldsymbol{k}}-g_j^{\boldsymbol{k}}\big(\tilde{\boldsymbol{\zeta}} \big)\Big),
$
and split it  into  two parts
\begin{equation}\label{tps}
\begin{array}{ll}
&\omega_jd_j^{\pm\langle j\rangle}= \epsilon \Big( \pm2
\textmd{i} \epsilon    \omega_j
\big(\dot{a}_j^{\pm\langle j\rangle}+(A\boldsymbol{a})_j^{\pm\langle
j\rangle}\big)-g_j^{\pm\langle j\rangle}\big(\tilde{\boldsymbol{\zeta}} \big)\Big),\\
&\boldsymbol{\omega}^{|\boldsymbol{k}| }d_j^{\boldsymbol{k}}=\epsilon^{[[\boldsymbol{k}]]}
\Big(\frac{4  s_{\langle
j\rangle+\boldsymbol{k}}s_{\langle
j\rangle-\boldsymbol{k}}}{h ^2\sinc(  h \omega_j) }\big(\boldsymbol{b}_j^{\boldsymbol{k}}+(B\boldsymbol{b})_j^
{\boldsymbol{k}}\big)-g_j^{\boldsymbol{k}}\big(\tilde{\boldsymbol{\zeta}}
\big)\Big).
\end{array}
\end{equation}
It is noted that the functions here are actually the results produced by  the $4N$-th  iteration. Reformulating
$g_j^{\pm\langle j\rangle}\big(\tilde{\boldsymbol{\zeta}}(t\epsilon)\big)$ and
$g_j^{\boldsymbol{k}}\big(\tilde{\boldsymbol{\zeta}}(t\epsilon)\big)$ in the form of $\mathbf{F},\mathbf{G}$ \eqref{FG} and then
inserting them from \eqref{Abstruct Pic ite} into \eqref{tps}, one gets
\begin{equation*}
\begin{array}{ll}
\omega_jd_j^{\pm\langle j\rangle}= 2 \omega_j \mu_j^{\pm\langle
j\rangle}\big( \big[ \dot{a}_j^{\pm\langle
j\rangle}\big]^{(4N)}-\big[\dot{a}_j^{\pm\langle
j\rangle}\big]^{(4N+1)}\big), \ \ \ & \nu_j^{\pm\langle
j\rangle}=\pm  \textmd{i}\epsilon^2,\\
\boldsymbol{\omega}^{|\boldsymbol{k}| }d_j^{\boldsymbol{k}}= \nu_j^{\boldsymbol{k}}\big([b_j^{\boldsymbol{k}}]^{(4N)}-
[b_j^{\boldsymbol{k}} ]^{(4N+1)}\big), \ \ \ &
\nu_j^{\boldsymbol{k}}=\epsilon^{[[\boldsymbol{k}]]}\frac{4
s_{\langle j\rangle+\boldsymbol{k}}s_{\langle j\rangle-\boldsymbol{k}}}{h ^2\sinc(  h \omega_j) }.
\end{array}
\end{equation*}
We have noticed that the form of these formulae is very closed  to that in
Section 6.9 of \cite{Cohen08-1} and then    \eqref{right hand esti} is obtained in a same way of \cite{Cohen08-1}.
The above analysis shows  $
 \norm{\boldsymbol{\delta}(t)}_{H^{s}}\leq C\epsilon^{N+1}$  for $h\leq  t\leq
\epsilon^{-1}.$

$\bullet$  \textbf{Defect in $\boldsymbol{\tilde{p}}$ and  initial values.} For another defect in the method $\boldsymbol{p}_n$:
 $$\boldsymbol{\rho}(t):=\frac{\boldsymbol{\tilde{p}}(t+h )+2h \Omega^2 \sinc( h \Omega) \boldsymbol{\tilde{q}}(t)-\boldsymbol{\tilde{p}}(t-h )}{h \big(\cos(h \Omega) +\sinc(  h \Omega)\big) }- \tilde{f}\big( \boldsymbol{\tilde{q}}(t)\big),$$
it can be expressed  as
  $$\begin{aligned}\boldsymbol{\rho}(t)=&\frac{\boldsymbol{\tilde{p}}(t+h )+2h \Omega^2 \sinc( h \Omega) \boldsymbol{\tilde{q}}(t)-\boldsymbol{\tilde{p}}(t-h )}{h \big(\cos(h \Omega) +\sinc(  h \Omega)\big) }-\frac{\boldsymbol{\tilde{q}}(t+h )-2\cos(h \Omega)\boldsymbol{\tilde{q}}(t)+\boldsymbol{\tilde{q}}(t+h )}{h ^2\sinc(  h \Omega) }+\boldsymbol{\delta}(t).\end{aligned}$$
By the result \eqref{methodpq rea}, we have $\boldsymbol{\rho}(t)=\boldsymbol{\delta}(t)$ and thus $\norm{\boldsymbol{\rho}(t)}_{H^{s}}\leq C\epsilon^{N+1}$.
For  the defect in the initial conditions
\eqref{initial pl} and \eqref{initial mi},  in a similar way  one gets $
\norm{\boldsymbol{q}_1-\boldsymbol{\tilde{q}}(h)}_{H^{1+s}}+\norm{\boldsymbol{p}_1-\boldsymbol{\tilde{p}}(h)}_{H^{s}}\leq
C\epsilon^{N+1}.$

$\bullet$  \textbf{Proof of the remainder \eqref{error MFE}.} Let $e^{\boldsymbol{q}}_n=\boldsymbol{\tilde{q}}(t_n)-\boldsymbol{q}_n,\ e^{\boldsymbol{p}}_n=\boldsymbol{\tilde{p}}(t_n)-\boldsymbol{p}_n$. We then have
 \begin{equation}\label{methodFerrormfe}
\begin{aligned}
&e^{\boldsymbol{q}}_{n+1}= 2\cos(h \Omega) e^{\boldsymbol{q}}_n-e^{\boldsymbol{q}}_{n-1}+\zeta_{n}^{\boldsymbol{q}}+\eta_{n}^{\boldsymbol{q}},\ \ e^{\boldsymbol{p}}_{n+1}= -2h \Omega^2 \sinc( h \Omega) e^{\boldsymbol{q}}_n+e^{\boldsymbol{p}}_{n-1}+\zeta_{n}^{\boldsymbol{p}}+\eta_{n}^{\boldsymbol{p}},
\end{aligned}
\end{equation}
where
 \begin{equation*}
\begin{aligned}
&\zeta_{n}^{\boldsymbol{q}}=h ^2\sinc(  h \Omega) \Delta \tilde{f}_n,\  \eta_{n}^{\boldsymbol{q}}=h ^2\sinc(  h \Omega) \boldsymbol{\delta}(t_n),\  \zeta_{n}^{\boldsymbol{p}}=h \big(\cos(h \Omega) +\sinc(  h \Omega)\big) \Delta \tilde{f}_n,\\
&  \eta_{n}^{\boldsymbol{p}}=h \big(\cos(h \Omega) +\sinc(  h \Omega)\big)\boldsymbol{\rho} (t_n),\ \ \ \Delta \tilde{f}_n= \tilde{f}(\boldsymbol{\tilde{q}} (t_n))-\tilde{f}(\boldsymbol{q}_n).
\end{aligned}
\end{equation*}Solving the first   equality yields
 \begin{equation}\label{rqf}
e_{n+1}^{\boldsymbol{q}}=W_{n-1}e_2^{\boldsymbol{q}}-W_{n-2}e_1^{\boldsymbol{q}}+ \sum_{k=1}^{n-1}W_{n-k}(\zeta_{k}^{\boldsymbol{q}}+\eta_{k}^{\boldsymbol{q}})
\end{equation} with $W_{n}:=\frac{\sin ( (n+1) h  \Omega )}{\sin( h  \Omega)}.$
In the analysis of this part,  we use the Lipschitz bound  (see \cite{Cohen08,Cohen08-1,Cohen08-0})
$$\norm{\tilde{f}(v)-\tilde{f}(w)}_{H^s}\lesssim \epsilon \norm{v-w}_{H^s}\ \textmd{for} \ v,w\in H^s\
\textmd{with}\
\norm{v}_{H^s}\lesssim \epsilon,\ \norm{w}_{H^s}\lesssim \epsilon,$$
which  is  yielded from the  Taylor expansion of the nonlinearity $\tilde{f}$  at $0$ and the fact that $H^s$ is a normed
algebra.
For the error $e_2^{\boldsymbol{q}}$, considering the one-step symmetric  scheme for the first formula of \eqref{method}, we have
 \begin{equation*}
\begin{aligned}\norm{e_2^{\boldsymbol{q}}}_{H^{1+s}}=&\norm{\cos(h \Omega) e_1^{\boldsymbol{q}}+h  \sinc( h \Omega) e_1^{\boldsymbol{p}}+h ^2/2 \sinc( h \Omega) \Big(\tilde{f}\big(\boldsymbol{\tilde{q}}(h)\big)-\tilde{f}\big( \boldsymbol{q}_1\big)\Big)}_{H^{1+s}}\\
 \lesssim&\norm{  e_1^{\boldsymbol{q}}}_{H^{1+s}}+   \norm{ e_1^{\boldsymbol{p}}}_{H^{s}}+h ^2/2 \epsilon \norm{ e_1^{\boldsymbol{q}}}_{H^{1+s}} \lesssim \epsilon^{N+1}.
  \end{aligned}
\end{equation*}
Therefore, we reach
  \begin{equation*}
\begin{aligned}
\norm{e_{n+1}^{\boldsymbol{q}}}_{H^{1+s}}
  \lesssim &(n-1) \epsilon^{N+1}+\sum_{k=1}^{n-1}h ^2 \big(  \norm{W_{n-k}\sinc(  h \Omega) (\Delta \tilde{f}_{k}+\boldsymbol{\delta}(t_k))}_{H^{1+s}}\big)\\
     \lesssim & \epsilon^{N}/h+\sum_{k=1}^{n-1}\big( h \norm{e_{k}^{\boldsymbol{q}}}_{H^{1+s}}+ h  \norm{\boldsymbol{\delta}(t_k)}_{H^{s}}\big)
      \lesssim  \epsilon^{N}/h+\sum_{k=1}^{n-1}  h   \norm{e_{k}^{\boldsymbol{q}}}_{H^{1+s}}.
\end{aligned}
\end{equation*}
This shows that $\norm{e_{n}^{\boldsymbol{q}}}_{H^{1+s}}  \lesssim \epsilon^{N}/h$ by Gronwall’s inequality.
The second equation of
\eqref{methodFerrormfe} gives
  \begin{equation*}
\begin{aligned}
\norm{e^{\boldsymbol{p}}_{n+1}}_{H^{s}}     \lesssim \norm{\sin( h \Omega) e^{\boldsymbol{q}}_n}_{H^{1+s}}+\norm{e^{\boldsymbol{p}}_{n-1}}_{H^{s}}+\norm{\zeta_{n}^{\boldsymbol{p}}}_{H^{s}}
+\norm{\eta_{n}^{\boldsymbol{p}}}_{H^{s}}.
\end{aligned}
\end{equation*}
From \eqref{rqf}, it follows that
  \begin{equation*}
\begin{aligned}
 &\norm{\sin( h \Omega) e_{n}^{\boldsymbol{q}}}_{H^{1+s}}\\
 \lesssim&  \norm{\sin ( (n-1) h  \Omega )e_2^{\boldsymbol{q}}}_{H^{1+s}}+ \norm{\sin ( (n-2) h  \Omega )e_1^{\boldsymbol{q}}}_{H^{1+s}} + \sum_{k=1}^{n-2} \norm{\sin ( (n-k+1) h  \Omega )(\zeta_{k}^{\boldsymbol{q}}+\eta_{k}^{\boldsymbol{q}})}_{H^{1+s}} \\ \lesssim& \epsilon^{N+1}
 + \sum_{k=1}^{n-2}h ^2 \epsilon \norm{e_{k}^{\boldsymbol{q}}}_{H^{1+s}}+ \sum_{k=1}^{n-2}  h\norm{\Omega^{-1}\boldsymbol{\delta}(t_k) }_{H^{1+s}} \lesssim   \epsilon^{N+1}
 + \sum_{k=1}^{n-2}h ^2 \epsilon \epsilon^{N}/h+ \sum_{k=1}^{n-2}  h\epsilon^{N+1}\lesssim  \epsilon^{N}.
\end{aligned}
\end{equation*}
Combining the above two estimates, one has
  \begin{equation*}
\begin{aligned}
\norm{e^{\boldsymbol{p}}_{n+1}}_{H^{s}}     \lesssim \norm{e^{\boldsymbol{p}}_{n-1}}_{H^{s}}+ \epsilon^{N}+h\epsilon \epsilon^{N}/h+
h\epsilon^{N+1} \lesssim \norm{e^{\boldsymbol{p}}_{n-1}}_{H^{s}}+ \epsilon^{N},
\end{aligned}
\end{equation*}
and then  $\norm{e_{n}^{\boldsymbol{p}}}_{H^{s}}  \lesssim \epsilon^{N-1}/h$ by Gronwall’s inequality. So far the remainder \eqref{error MFE} has been shown which completes the whole proof of Lemma \ref{MFE lem}.
\end{proof}

\subsubsection{Almost-invariants} \label{subsec:alm inv}
 In
this section, we shall   show that the modulated Fourier expansion
\eqref{MFE-trigonometric} has three almost-invariant  which are closed to the
 energy,  actions
and   momentum of \eqref{prob}.
\begin{mylemma}\label{invariant12 thm}
For the modulated Fourier expansion \eqref{MFE-trigonometric}, we collect all the coefficient functions in
$\vec{\boldsymbol{\zeta}}=\big(\boldsymbol{\zeta}^{\mathbf{k}}\big)_{\mathbf{k}}$.
Under the conditions of Theorem \ref{main theo1},    there exist  three functions
$\mathcal{H}[\vec{\boldsymbol{\zeta}}],\ \mathcal{J}_{l}[\vec{\boldsymbol{\zeta}}],\ \mathcal{K}[\vec{\boldsymbol{\zeta}}]$ such that
\begin{equation}\label{AIR1}
\begin{aligned}
  \left|\frac{d}{d \tilde{t}}
\mathcal{H}[\vec{\boldsymbol{\zeta}}](\tilde{t})\right|\leq C\epsilon^{N+1},\
\sum\limits_{l=1}^M \omega_l^{2s+1} \left|\frac{d}{d \tilde{t}}
\mathcal{J}_{l}[\vec{\boldsymbol{\zeta}}](\tilde{t})\right|\leq C\epsilon^{N+1},\ \left|\frac{d}{d \tilde{t}} \mathcal{K}[\vec{\boldsymbol{\zeta}}](\tilde{t})\right|\leq
C(\epsilon^{N+1}+\epsilon^{2}M^{-s+1}),
\end{aligned}
\end{equation}
where $h\leq\tilde{t}=\epsilon t \leq1.$
 Moreover, they have the following relations along the
numerical solution
\begin{equation}\label{AIR2}
\begin{aligned}
&\mathcal{H}[\vec{\boldsymbol{\zeta}}](\epsilon t_n)=H_M(\boldsymbol{q}^n ,
\boldsymbol{p}^n)+\mathcal{O}(\epsilon^3), \quad  \mathcal{J}_{l}[\vec{\boldsymbol{\zeta}}](\epsilon t_n)=J_{l}(\boldsymbol{q}^n ,
\boldsymbol{p}^n)+\gamma_l(t_n)\epsilon^3,\\
&\mathcal{K}[\vec{\boldsymbol{\zeta}}](\epsilon t_n)=K(\boldsymbol{q}^n ,
\boldsymbol{p}^n)+\mathcal{O}(\epsilon^3)+\mathcal{O}(\epsilon^2 M^{-s}),
\end{aligned}
\end{equation}
where $1\leq n  \leq\frac{1}{\epsilon h}$, $J_{l}: = I_l + I_{-l}$ for  $0 < l <
M,  J_0: = I_0,  J_M := I_M$ and  the constants are independent of $\epsilon, M, h $ and
$n$.
\end{mylemma}
\begin{proof}
The proof will only be presented on the energy conservation  for simplicity, and the statement on actions
and   momentum can be derived in the same way as \cite{W23}.

$\bullet$  \textbf{Proof of \eqref{AIR1}.}
Define $\boldsymbol{\boldsymbol{z}}^{\boldsymbol{k}}(t)=\mathrm{e}^{\mathrm{i}(\boldsymbol{k}\cdot\boldsymbol{\omega}) t}\boldsymbol{\zeta}^{\boldsymbol{k}}(\epsilon t)$ and $\vec{\boldsymbol{z}}(t)=\big(z^{\mathbf{k}}(t)\big)_{\mathbf{k}}$.
According to the proof of   Theorem \ref{MFE lem},  the defect
formula \eqref{djk} can be reformulated as
\begin{equation}\label{djk pote}
\begin{aligned}
&\frac{1}{h ^2\sinc(  h \omega_j)}\tilde{L}^{\boldsymbol{k}}(\epsilon
h  D)
z_j^{\boldsymbol{k}}(t)+\nabla_{-j}^{-\boldsymbol{k}}\mathcal{U}
(\vec{\boldsymbol{z}}(t))=\mathrm{e}^{\mathrm{i}(\boldsymbol{k}\cdot\boldsymbol{\omega}) t}d_j^{\boldsymbol{k}}(\epsilon t),
\end{aligned}
\end{equation}
where $\nabla_{-j}^{-\boldsymbol{k}}\mathcal{U}(\vec{\boldsymbol{z}})$ is the partial derivative
with respect to $\zeta_{-j}^{-\boldsymbol{k}}$ of the extended potential
(\cite{Cohen08-1})
\begin{equation*}
\begin{aligned}
&\mathcal{U}(\vec{\boldsymbol{z}})=\sum\limits_{l=-N}^N\mathcal{U}_l(\vec{\boldsymbol{z}}),\ \
 \mathcal{U}_l(\vec{\boldsymbol{z}})=\sum\limits_{m= 2}^N
\frac{U^{(m+1)}(0)}{(m+1)!}
\sum\limits_{\boldsymbol{k}^1+\cdots+\boldsymbol{k}^{m+1}=0}\sum\limits_{j_1+\cdots+j_{m+1}=2Ml}'
\big(z_{j_1}^{\boldsymbol{k}^1}\cdot\ldots\cdot
z_{j_{m+1}}^{\boldsymbol{k}^{m+1}}\big).
\end{aligned}
\end{equation*}
Here $U$ is the potential appearing in the energy function and $\norm{\boldsymbol{k}^i}\leq
2N$ and $| j_i| \leq M$ for $i=1,2,\ldots,m+1$.

The almost-invariant is obtained by considering
\begin{equation*}
\begin{aligned}
& \epsilon\frac{d}{d\tilde{t}} \mathcal{U}(\vec{\boldsymbol{z}}(t))=\frac{d}{dt} \mathcal{U}(\vec{\boldsymbol{z}}(t))
= \sum\limits_{\norm{\boldsymbol{k}}\leq K}\sum\limits_{|j|\leq M}'
\dot{z}_{-j}^{-\boldsymbol{k}}(t) \nabla_{-j}^{-\boldsymbol{k}}\mathcal{U}(\vec{\boldsymbol{z}}(t)).
\end{aligned}
\end{equation*}
 According to \eqref{djk pote} and the above formula,
we have
\begin{equation*}
\begin{aligned}
 \epsilon\frac{d}{d\tilde{t}} \mathcal{U}(\vec{\boldsymbol{\zeta}}(t))
=  &  \sum\limits_{\norm{\boldsymbol{k}}\leq K}\sum\limits_{|j|\leq M}' \big
(\textmd{i}(\boldsymbol{k}\cdot\boldsymbol{\omega})\zeta_{-j}^{-\boldsymbol{k}}(\tilde{t})+
 \epsilon\dot{\zeta}_{-j}^{-\boldsymbol{k}}(\tilde{t}) \big)^\intercal
\Big (d_j^{\boldsymbol{k}}(\tilde{t})-\frac{\tilde{L}^{\boldsymbol{k}}(\epsilon
h  D +i \epsilon
h (\boldsymbol{k}\cdot\boldsymbol{\omega})) }{h ^2\sinc(  h \omega_j)}\zeta_j^{\boldsymbol{k}}(\tilde{t})\Big).
\end{aligned}
\end{equation*}
By the expansion  \eqref{new L2r} of the operator $\tilde{L}^{\boldsymbol{k}}$,
it is known that $\tilde{L}^{\boldsymbol{k}}(\epsilon
h  D +i \epsilon
h (\boldsymbol{k}\cdot\boldsymbol{\omega})) \zeta_j^{\boldsymbol{k}}$ has the following form
\begin{equation*}
\begin{aligned}
\tilde{L}^{\boldsymbol{k}}(\epsilon
h  D +i \epsilon
h (\boldsymbol{k}\cdot\boldsymbol{\omega})) \zeta_j^{\boldsymbol{k}}=\varsigma_0\zeta_j^{\boldsymbol{k}}+\textmd{i}\varsigma_1\dot{\zeta}_j^{\boldsymbol{k}}+\varsigma_2\ddot{\zeta}_j^{\boldsymbol{k}}
+\textmd{i}\varsigma_3 (\zeta_j^{\boldsymbol{k}})^{(3)}+\cdots
\end{aligned}
\end{equation*}
with all  $\varsigma_m\in \mathbb{R}$. Considering $\zeta_{-j}^{-\boldsymbol{k}}=\overline{\zeta_j^{\boldsymbol{k}}}$ and the formulae from  p. 508 of \cite{hairer2006},
it is obtained that the part
$$\sum\limits_{\norm{\boldsymbol{k}}\leq K}\sum\limits_{|j|\leq M}' \big
(\textmd{i}(\boldsymbol{k}\cdot\boldsymbol{\omega})\zeta_{-j}^{-\boldsymbol{k}}(\tilde{t})+
 \epsilon\dot{\zeta}_{-j}^{-\boldsymbol{k}}(\tilde{t}) \big)^\intercal
 \frac{\tilde{L}^{\boldsymbol{k}}(\epsilon
h  D +i \epsilon
h (\boldsymbol{k}\cdot\boldsymbol{\omega}))}{h ^2\sinc(  h \omega_j)} \zeta_j^{\boldsymbol{k}}(\tilde{t})$$ is a total derivative of
a function w.r.t $\tilde{t}$. Thus,  there is a function $ \mathcal{H}[\vec{\zeta}](\tilde{t})$
such that
\begin{equation*}
\begin{aligned}
 \frac{d}{d \tilde{t}} \mathcal{H}[\vec{\boldsymbol{\zeta}}](\tilde{t})=&
\sum\limits_{\norm{\boldsymbol{k}}\leq K}\sum\limits_{|j|\leq M}' \big
(\textmd{i}(\boldsymbol{k}\cdot\boldsymbol{\omega})\zeta_{-j}^{-\boldsymbol{k}}(\tilde{t})+
 \epsilon\dot{\zeta}_{-j}^{-\boldsymbol{k}}(\tilde{t}) \big)^\intercal d_j^{\boldsymbol{k}}(\tilde{t}).
\end{aligned}
\end{equation*}
According to the bound \eqref{bound modu func} for the functions $\zeta_{-j}^{-\boldsymbol{k}},\dot{\zeta}_{-j}^{-\boldsymbol{k}}$ and the bound \eqref{right hand esti} for $d_j^{\boldsymbol{k}}$,
the statement   \eqref{AIR1} is got by the
Cauchy-Schwarz inequality immediately.

$\bullet$ \textbf{Proof of \eqref{AIR2}.}
With the expansion  \eqref{new L2r} of the operator $\tilde{L}^{\boldsymbol{k}}$,   the   bound \eqref{bound modu func} for the functions $\zeta_{-j}^{-\boldsymbol{k}}$ and the relation \eqref{impor-zetaeta} between $\eta_j^{\pm\langle j\rangle}$  and $\zeta_j^{\pm\langle j\rangle}$, the expression of the almost-invariant $ \mathcal{H}[\vec{\boldsymbol{\zeta}}]$ is obtained as
$
  \mathcal{H}[\vec{\boldsymbol{\zeta}}](\tilde{t})
 =\sum\limits_{|j|\leq M}' \omega_j^2\Big(
 |\zeta_j^{\langle j\rangle}|^2+ |\zeta_j^{-\langle j\rangle}|^2\Big)(\tilde{t})+\mathcal{O}(\epsilon^3).$
Meanwhile, inserting the modulated Fourier expansion \eqref{MFE-trigonometric} into the energy $H_M$ leads to
$$\begin{aligned}
H_M(\boldsymbol{q}_n,\boldsymbol{p}_n)&=H_M(\boldsymbol{\tilde{q}}(\tilde{t}),\boldsymbol{\tilde{p}}(\tilde{t})) +\mathcal{O}(\epsilon^{N-1}/h) =\sum\limits_{|j|\leq M}' \omega_j^2\Big(
 |\zeta_j^{\langle j\rangle}|^2+ |\zeta_j^{-\langle j\rangle}|^2\Big)+\mathcal{O}(\epsilon^3)+\mathcal{O}(\epsilon^{N-1}/h),\end{aligned}$$
where the results \eqref{jth TMFE} and \eqref{error MFE} are used. The above two results  and the arbitrarily $N$ immediately show \eqref{AIR2}.
\end{proof}
\subsubsection{Proof of long time conservations}
 \begin{proof}The above statements are on the interval $[0,\epsilon^{-1}]$ and they can be extended to a long interval
by patching together many short intervals of length $\epsilon^{-1}$ which is the same as  \cite{Cohen08,Cohen08-1,Cohen08-0}.
Then with the above preparations given in this section,  it is arrived that
\begin{equation*}
\begin{aligned}
\frac{H_M(\boldsymbol{q}_n,\boldsymbol{p}_n)}{\epsilon^2}&= \frac{ \mathcal{H}[\vec{\boldsymbol{\zeta}}](\epsilon t_n)}{\epsilon^2} + \mathcal{O}(\epsilon)=\frac{ \mathcal{H}[\vec{\boldsymbol{\zeta}}](\epsilon t_{n-1})}{\epsilon^2}+h  \epsilon\mathcal{O}(\epsilon^{N-1})  +\mathcal{O}(\epsilon)  =\cdots \\&=\frac{ \mathcal{H}[\vec{\boldsymbol{\zeta}}](\epsilon t_{1})}{\epsilon^2}+(n-1)h  \epsilon\mathcal{O}(\epsilon^{N-1})  +\mathcal{O}(\epsilon)
 =\frac{H_M(\boldsymbol{q}_0,\boldsymbol{p}_0)}{\epsilon^2}  +\mathcal{O}(h^3) +\mathcal{O}(\epsilon)
\end{aligned}
\end{equation*}
as long as $nh  \leq \epsilon^{-N+1}$, where the last
equation comes from the accuracy of the starting value $\boldsymbol{q}_1,\boldsymbol{p}_1$  \eqref{start value}  which has been well researched in \cite{LI23}.
This shows the statement of energy  on the interval $[0,\epsilon^{-N+1}]$ and the result of actions
and   momentum is obtained with the same arguments as above and \cite{Cohen08-1}. The proof of Theorem \ref{main theo1} is complete.
 \end{proof}



\end{document}